\documentclass[10pt]{smfart}

\author{Daichi Kohmoto}
\address{Nagoya University, Graduate School of Mathematics \\
Furo, Chikusa, Nagoya 464-8602, Japan}
\email{kohmoto.daichi@math.nagoya-u.ac.jp}

\usepackage{amsmath,a4wide}
\usepackage{amssymb}
\usepackage{rsfs}
\usepackage[latin1]{inputenc}
\usepackage{euscript}
\usepackage{xypic}
\usepackage{smfthm}
\usepackage{devanagari}
\usepackage{color}
\usepackage{graphicx}
\SwapTheoremNumbers

\renewcommand\thesubsubsection{\textbf{\thesubsection.\arabic{subsubsection}}} 
\renewcommand\thesubsection{\textbf{\thesection.\arabic{subsection}}}

\newtheorem{proposition}[subsubsection]{Proposition} 

\newtheorem{theorem}[subsubsection]{Theorem}

\newtheorem{theorem3}{Theorem}
\newtheorem{lemma}[subsubsection]{Lemma}

\newtheorem{lemma3}[subsubsection]{Key-Lemma}

\theoremstyle{definition}

\def\Coker{{\rm{Coker}}}

\def\Hom{{\rm{Hom}}}
\def\Ext{{\rm{Ext}}}

\def\Spec{{\rm{Spec}}}

\def\x{{\mathcal{x}}}

\def\Ker{{\rm{Ker}}}

\def\HH{\mathbb{H}}

\def\Gal{{\rm{Gal}}}
\def\G{{\mathbb{G}}}

\def\mot{{\rm{mot}}}
\def\et{{\rm{\acute{e}t}}}
\def\Wet{{{\text{W-}}{\rm{\acute{e}t}}}}

\def\2tor{{{2 \overset{\_}{ } \rm{tor}}}}

\def\bX{{\overline{X}}}

\def\Pic{{\rm{Pic}}}

\def\Pic{{\rm{Pic}}}
\def\O{{\mathscr{O}}}

\def\F{{\mathbb{F}}}
\def\Q{{\mathbb{Q}}}
\def\Z{{\mathbb{Z}}}

\def\P{{\mathbb{P}}}
\def\ur{{\rm{ur}}}

\def\T{{\mathscr{T}}}
\def\crys{{\rm{crys}}}
\def\e{{\varepsilon}}

\def\rat{{\sf{rat}}}
\def\hom{{\sf{hom}}}
\def\num{{\sf{num}}}

\def\L{{\mathbb{L}}}
\def\oF{{\overline{\mathbb{F}}}}

\def\rank{{\rm{rank}}}
\def\NS{{\rm{NS}}}

\def\P{{\mathbb{P}}}

\def\r{{\varrho}}
\def\mod{{\rm{mod}}}
\def\cont{{\rm{cont}}}
\def\ur{{\rm{ur}}}
\def\tors{{\rm{tors}}}

\def\Fr{{\sf{Fr}}}
\def\ord{{\rm{ord}}}
\def\hZ{{\widehat{\mathbb{Z}}}}
\def\cont{{\rm{cont}}}

\def\Br{{\rm{Br}}}

\def\L{{\mathbb{L}}}
\def\oX{{\overline{X}}}

\makeindex

\begin{document}    

\centerline{\bf A generalization of the Artin-Tate formula for fourfolds}


\bigskip
\bigskip
\centerline{Daichi Kohmoto}

\bigskip
\bigskip{\small 
\begin{quotation}
\noindent
{\bf  \sc Abstract.} 
We give a new formula for the special value at $s=2$ of the Hasse-Weil zeta function for smooth projective fourfolds under some assumptions (the Tate \& Beilinson conjecture, finiteness of some cohomology groups, etc.). Our formula may be considered as a generalization of the Artin-Tate(-Milne) formula for smooth surfaces, and expresses the special zeta value almost exclusively in terms of inner geometric invariants such as higher Chow groups (motivic cohomology groups). Moreover we compare our formula with Geisser's formula for the same zeta value in terms of Weil-\'etale motivic cohomology groups, and as a consequence (under additional assumptions) 
we obtain some presentations of weight two Weil-\'etale motivic cohomology groups in terms of higher Chow groups and unramified cohomology groups. 
\end{quotation}}

\bigskip
\tableofcontents

\section{Introduction}

The Hasse-Weil zeta function is one of the most fundamental and important objects in arithmetic algebraic geometry. 
Let $X$ be a smooth projective variety of dimension $d$ defined over a finite field $\F_q$. 
This zeta function $\zeta (X,s)$ is defined for such $X$ by counting rational points of $X$:

\centerline{$\displaystyle \zeta(X,s):=Z(X,q^{-s})$ \ \ with \ \ $\displaystyle Z(X,t):=\exp \left( 
\sum_{n\geq 1} \frac{\# X(\F_{q^n})}{n} t^n \right),$}

\noindent
and is expected to contain much information about $X$ itself, which has both arithmetic and 
geometric aspects. 
Arithmetically, 
the Hasse-Weil zeta function is a natural generalization of 
the Riemann zeta function, and 
geometrically, the Weil conjecture 
asserts that this zeta function `{\em knows the topology}' of $X$. 
The Weil conjecture was proved by Grothendieck and Deligne (\cite{SGA4},\cite{Deligne1}), 
who developed the theory of \'etale cohomology to solve this conjecture. 
Since then, this powerful machinery has been a major tool in the study of Hasse-Weil zeta functions.  

In this paper, we are concerned with {\em a special value of Hasse-Weil zeta function} 
that we define as follows: let $n$ be a natural number, 
and let $\rho_n$ be the order of $\zeta(X,s)$ at $s=n$. Then  
we define the {\em special value of $\zeta(X,s)$ at $s=n$} as  
$$\zeta^*(X,n) := \lim_{s \to n} \zeta(X,s) (1-q^{n-s})^{-\rho_n}\ \ \ \ (0\leq n\leq d),$$
which is a non-zero rational number by the Weil conjecture. 

In the case $d\geq 2$, the following theorem due to Tate, Artin (\cite[Theorem 5.2]{Tate1}), and Milne (\cite[Theorem 6.1]{Milne0}) is the first 
striking result on the study of zeta values. 
Let $X$ be a smooth projective surface defined over a 
finite field $\F_q$ of characteristic $p>0$, and we denote $\overline{X} := X \otimes_{\F_q} \oF_q$. 
It is well-known that the torsion part of the Picard group 
$\Pic(X)_{\tors}$ is finite and that the map 
$\iota:\Pic(X) \to \Hom(\Pic(X),\Z)$ induced by the intersection pairing 
$\Pic(X)\times \Pic(X) \to \Z$ has finite kernel and cokernel. 
Assume one of the following equivalent conditions ($\star$) for a simple prime number $\ell \neq p$:
$$(\star) \left\{\begin{array}{l}
\displaystyle (1) \ \dim_{\Q_\ell} H_\et^2(\oX,\Q_\ell(1))^\Gamma = \rank_\Z \NS(X)\ {\rm for\ some\ }\ell\ {\rm where\ }\Gamma := \Gal(\oF_q/\F_q).\\
\displaystyle (2) \ \ell\text{-}{\rm torsion\ subgroup\ of\ the\ Brauer\ group\ }\Br(X)\{\ell\}\ {\rm is\ finite}.\\
\displaystyle (3) \ \rank_\Z \NS(X)\ {\rm is\ equal\ to\ the\ multiplicity\ of\ }q\ {\rm as\ a\ reciprocal\ root\ of\ }P_2(X,t)
\end{array}\right.$$
Here $\NS(X)$ denotes the N\'eron-Severi group of $X$. 
Then the Brauer group $\Br(X)=H_\et^2(X,\G_m)$ 
of $X$ is finite, and moreover the following 
formula holds:

\medskip
\noindent
$({\rm A})$ \hfill $\displaystyle \zeta^*(X,1)=(-1)^{S(1)} \cdot q^{\chi(X,\O_X)} \cdot 
\frac{ \left| \Pic(X)_\tors \right|^2}{(q-1)^2 \cdot D \cdot \left| \Br(X) \right|}.$ \hfill { }\\

\medskip
\noindent
Here $D$ is the order of the cokernel of the map $\iota: \Pic(X) \to \Hom(\Pic(X),\Z)$, 
and $S(1)$ and $\chi(X,\O_X)$ are the numbers defined by
$$S(1) := \sum_{a>4} \rho_{\frac{a}{2}}
,\ \ \ \rho_{\frac{a}{2}} := \ord_{s=\frac{a}{2}} \zeta (X,s),\ \ \ \ \ \chi(X,\O_X)=\sum_{0\leq j\leq 2} (-1)^j \dim_{\F_q} H^j(X,\O_X),$$
where $\displaystyle \rho_{\frac{a}{2}}$ is the order of $\zeta(X,s)$ at 
$\displaystyle s=a/2$, and $a$ runs through all integers at least three. 
The formula (A) is called {\em the Artin-Tate formula}. 
The sign $(-1)^{S(1)}$ is due to Kahn and Chambert-Loir, 
and the equivalence of the above 
three conditions $(\star)$ was proved by Tate (\cite[Theorem 5.1]{Tate1}). 

After this pioneering work, 
B\'ayer, Neukirch, Schneider 
and Milne generalized the Artin-Tate formula 
to arbitrary $d$-dimensional smooth projective variety $X$ and weight $n$ 
under the assumption 
of Tate's partial 
semisimplicity conjecture ${\boldsymbol{SS}^{\boldsymbol{n}}}(X,\ell)$ for all $\ell$ 
(see $\x$\ref{conjecture} for this conjecture) 
(\cite{BN}, \cite{Schneider}, and \cite[Theorem 0.1]{Milne1}). 
For a prime number $\ell$, let $H_\et^i(X,\Z_\ell(n))$ be the $\ell$-adic \'etale cohomology group defined as follows:
$$H_\et^i(X,\Z_\ell(n)) := \lim_{\stackrel{\longleftarrow}{{\small \nu}}} 
H_\et^i(X,\mu_{\ell^\nu}^{\otimes n})\ {\rm for\ }\ell \neq p,\ \ \ 
H^i_\et(X,\Z_p(n)) := \lim_{\stackrel{\longleftarrow}{{\small r \geq 1}}} H^{i-n}_\et (X,W_r\Omega^n_{X,\log}),$$
where $\mu_{\ell^\nu}$ denotes the \'etale sheaf of $\ell^\nu$-th roots of unity, and 
$W_r\Omega^n_{X,\log}$ is the \'etale sheaf of the logarithmic part of the Hodge-Witt sheaf 
$W_r\Omega^n_X$ (\cite{Illusie}). 
We define $H_\et^i(X,\hZ(n))$ as follows:
$$H^i_\et(X,\hZ(n)) := \prod_{{\rm all\ primes}\ \ell} H^i_\et(X,\Z_\ell(n)).$$
By celebrated theorems of Deligne (\cite[Th\'eor\`eme 1.6]{Deligne1}) 
and Gabber (\cite{Gabber}), 
$H_\et^i(X,\hZ(n))$ is finite for all $i\neq 2n,2n+1$, and 
$H_\et^{2n}(X,\hZ(n))_\tors$ is finite. 
Tate's partial semisimplicity conjecture $\boldsymbol{SS}^{\boldsymbol{n}}(X,\ell)$ for all $\ell$ implies the finiteness of $\Ker(\e^{2n})$ and $\Coker(\e^{2n})$ 
where $$\e^{2n} : H_\et^{2n}(X,\hZ(n)) \to H_\et^{2n+1}(X,\hZ(n))$$
is the cup product with the canonical element $1 \in \hZ \simeq H_\et^1(\Spec~\F_q,\hZ)$ 
(\cite[$\x$6]{Milne1}). 
Using these results, Milne gave the following general formula for $\zeta^*(X,n)$:
$$\zeta^*(X,n)=(-1)^{S(n)} \cdot q^{\chi(X,\O_X,n)} \cdot \prod_{i\neq 2n,2n+1} \left| H_\et^i(X,\hZ(n)) \right|^{(-1)^i} \cdot \frac{|\Ker(\e^{2n})|}{|\Coker(\e^{2n})|}.$$
Here $S(n)$ and $\chi(X,\O_X,n)$ are defined as 
$$S(n) := \sum_{a>2n} \rho_{\frac{a}{2}},\ \ \chi(X,\O_X,n) := \sum_{i,j} (-1)^{i+j} (n-j) \dim_{\F_q} H^i(X,\Omega^j)\ \ \ 
(0\leq i\leq d,\ 0\leq j\leq n),$$
and the presentation of $S(n)$ is due to Kahn and Chambert-Loir 
(\cite[Remarque 3.11]{Kahn}, \cite[Conjecture 66]{Kahn3}). 
The above general formula 
shows that zeta values at integers are expressed in terms of \'etale cohomology groups.
In this paper, we would like to find a formula for $\zeta^*(X,n)$ 
using inner geometric invariants of $X$, 
{\em i.e.,} higher Chow groups (motivic cohomology groups). 

The first result in this direction is due to Saito and Sato (\cite[Theorem B.1]{SaitoSato}). 
They proved a formula for $\zeta^*(X,2)$ for a smooth projective 
threefold $X$ based on Milne's formula assuming the 
finiteness of $\Br(X)$ and the unramified cohomology group $H_\ur^3(k(X),\Q/\Z(2))$. 
Here, $H_\ur^3(k(X),\Q/\Z(2))$ is defined as the 
kernel of the following boundary map $\partial$ in the localization spectral sequence of \'etale cohomology groups (cf. \cite[p.21 (3.6)]{CT}): 
$$\partial : H_\et^3(\Spec~k(X),\Q/\Z(2)) \to \bigoplus_{x\in X^{(1)}} H_{x,\et}^4(\Spec~\O_{X,x},\Q/\Z(2)),$$
\noindent
where $k(X)$ denotes the function field of $X$, and 
$X^{(1)}$ denotes the set of points of $X$ of codimension one. 
(See $\x$\ref{etale} below for $\Q/\Z(2)$.)

\bigskip
The first aim of this paper is to 
formulate a formula described almost exclusively in terms of Zariski topology 
for smooth projective fourfolds assuming the following two conjectures 
(see $\x$\ref{conjecture} for these conjectures). 
$$(\star \star) \left\{\begin{array}{l}
({\rm i})\ CH^2(X)\otimes \Q\  {\rm is\ finite\text{-}dimensional\ over\ }\Q,\ {\rm and\ the\ order\ }\rho_2\ {\rm of\ the\ pole\ of\ }\zeta(X,s)\ {\rm at}\\
\ \ \ \ s=2\ {\rm is\ equal\ to\ }\dim_\Q \left( CH^2(X)\otimes \Q \right).\ {\rm We\ call\ this\ hypothesis\ }the\ Tate\ \&\ Beilinson\\
\ \ \ \ conjecture\ for\ X\ {\rm and\ denote\ it\ by\ }
{\boldsymbol{TB}}^{\boldsymbol{2}}(X).\\
({\rm ii})\ CH^2(X)\ {\rm is\ finitely\ generated}. 
\end{array}
\right.\ \ \ \ \ \ $$

\noindent
There are examples of smooth projective $4$-folds satisfying $(\star \star)$ 
which are constructed using Soul\'e's result (\cite[Th\'eor\`eme 3]{Soule}) for $3$-dimensional case. 
We present an example; 
let $A$ be a $3$-dimensional abelian variety defined over $\F_q$, then 
$X=A\times_{\F_q} \P_{\F_q}^1$ is a $4$-dimensional variety satisfying $(\star \star)$, 
{\em i.e.,} $\boldsymbol{TB}^{\boldsymbol{2}}(X)$ and the finite generation of $CH^2(X)$. 
We will discuss the conjectures in $(\star \star)$ and their consequences in $\x$\ref{pre}. 
The first main result of this paper is as follows. 

\begin{theorem3}\label{conjecturetheorem}
Let $X$ be a $4$-dimensional smooth projective geometrically integral variety 
defined over a finite field $\F_q$. 

\begin{enumerate}
\item[(a)] 
If ${\boldsymbol{TB}^{\boldsymbol{2}}(X)}$ holds and $CH^2(X)$ is finitely generated, 
then $H_\ur^3(k(X),\Q/\Z(2))$ is finite. 

\item[(b)] 
${\boldsymbol{TB}^{\boldsymbol{2}}(X)}$ implies that $H_\et^{5}(X,\Z(2))$ is finite, 
where $\Z(2)$ is the \'etale sheafification of Bloch's cycle complex on $X$ 
(see $\x$\ref{finitenessresult} for more details).

\item[(c)]
${\boldsymbol{TB}^{\boldsymbol{2}}(X)}$ implies that the intersection pairing 
$$CH^2(X) \times CH^2(X) \to CH^4(X)\simeq CH_0(X) \overset{\deg}{\longrightarrow} \Z$$ is non-degenerate when tensored with $\Q$.

\item[(d)] 
$CH^2(X,i)_\tors$ is finite for $i=1,2,3$ and zero for $i\geq 4$.  

\end{enumerate}
\end{theorem3}

\noindent
Based on this result, our formula is as the following:

\begin{theorem3}\label{mainthm2}
Let $X$ be a $4$-dimensional 
smooth projective geometrically integral variety defined over a finite field $\F_q$. 
Assume the above two conjectures in $(\star \star)$ for $X$. 
Then the following formula holds$:$

\medskip
\noindent
$({\rm B})$ \hfill 
$\displaystyle \zeta^*(X,2)=(-1)^{S(2)} \cdot q^{\chi(X,\O_X,2)} \cdot \frac{|H_\ur^3(k(X),\Q/\Z(2))|^2}{|H^5_\et(X,\Z(2))| \cdot R_1} \cdot \prod_{i=0}^3 |CH^2(X,i)_\tors|^{2 \cdot (-1)^i}.$ \hfill { }\\

\medskip
\noindent
Here $R_1$ is the order of the cokernel of the map 
$$\theta : CH^2(X) \to \Hom( CH^2(X),\Z)$$
induced by the intersection pairing ($R_1$ is well-defined by Theorem \ref{conjecturetheorem} (c)), and $\chi(X,\O_X,2)$ is defined as  
$$\chi(X,\O_X,2) :=  \sum_{i,j} (-1)^{i+j}(2-i) \dim_{\F_q} H^j(X,\Omega_X^i)\ \ \ \ (0\leq i\leq 2,\ 0\leq j \leq 4).$$
\end{theorem3}

\noindent
We call this result {\em the higher Chow formula} for $\zeta^*(X,2)$, 
which is considered as a generalization of the Artin-Tate formula 
for the following two reasons.  
First, we consider the special value at the half-dimension-point, and second, 
comparing our formula with the Artin-Tate formula, there are some interesting analogies: 
the term $D$ corresponds to $R_1$ in our formula, 
$\Pic(X)_\tors \cdot (q-1)^{-2}$ corresponds to the higher Chow terms, 
and $|\Br(X)|$ corresponds to $|H_\et^5(X,\Z(2))| \cdot |H_\ur^3(k(X),\Q/\Z(2))|^{-2}$, and both numbers are squares or twice squares.

Concerning the case of $\dim X=4$ such as in our case, 
Milne have already showed two formulas for $\zeta^*(X,2)$ in 
\cite[Theorem 0.6]{Milne1} and \cite[Theorem 6.6]{Milne2}. 
But both of two formulas are not expressed in terms of inner geometric invariants. 
In addition, his results are based on 
an additional condition of surjectivity of the cycle class map of codimension two 
(cf. $\boldsymbol{T}^{\boldsymbol{n}}(X,\ell)$ in $\x 2.1$), which 
we do not assume in our main theorem. 
Our result gives a new zeta value formula in this sense.

As for $\zeta^*(X,1)$ and $\zeta^*(X,3)$, we unfortunately do not have 
fully a developed theory of a cycle class map of codimension three.  
Therefore, we do not have anything to say these special values in this paper.

\medskip
Our second main theorem concerns the comparison of our higher Chow formula (B) with 
the following formula due to Geisser in terms of Weil-\'etale motivic theory: 
$$\zeta^*(X,2)=(-1)^{S(2)} \cdot q^{\chi(X,\O_X,2)} \cdot \frac{1}{R_2} \cdot \prod_{i=0}^9 \left| H_\Wet^i(X,\Z(2))_\tors \right|^{(-1)^i},$$
under the assumptions (in addition to $(\star \star)$) 
that the Weil-\'etale motivic cohomology groups 
$H_\Wet^i(X,\Z(n))$ is finitely generated for every $i$ and $n$ 
(\cite[Theorem 9.1]{Geisser1}) and that $CH^2(X,1)\otimes \Q=0$ which is a part of the Parshin conjecture. 
Here $R_2$ is the order of the cokernel of the map $\vartheta : 
H_\Wet^4(X,\Z(2)) \to \Hom(H_\Wet^4(X,\Z(2)),\Z)$ induced by the following 
pairing  
$$H_\Wet^{4}(X,\Z(2)) \times H_\Wet^4(X,\Z(2)) \to H_\Wet^8(X,\Z(4)) \overset{- \cup e}{\to} H_\Wet^9(X,\Z(2)) \overset{\deg}{\to} \Z,$$
where $e$ is the canonical element corresponding to $1$ 
under the isomorphism of $H_\Wet^1(\Spec~\F_q,\Z) \overset{\simeq}{\longrightarrow} 
\Z$. We compare this formula with our higher Chow formula (B). 
More precisely, we wish to clarify which term of the 
higher Chow formula corresponds to which 
term of Geisser's formula. 
For this purpose, we begin with comparing higher Chow groups with 
Weil-\'etale motivic cohomology groups. By the duality arguments and 
standard facts on the cycle class map of codimension two, 
we get the following dictionary assuming $\boldsymbol{TB}^{\boldsymbol{2}}(X)$: 
$$H_\Wet^0(X,\Z(2))_\tors=0,\ \ \ H_\Wet^i(X,\Z(2))_\tors=
CH^2(X,4-i)_\tors \ \ (i=1,2,3),$$
$$H_\Wet^i(X,\Z(2))=\left( CH^2(X,i-6)_\tors \right)^*\ (i=7,8,9),\ H_\Wet^i(X,\Z(2))=0\ (i\geq 10).$$
The cases $i=4,5$ and $6$ are not so easy, 
but we have the following two facts by duality arguments: 
(i) under $\boldsymbol{TB}^{\boldsymbol{2}}(X)$ both 
$H_\Wet^4(X,\Z(2))_\tors$ and $H_\Wet^6(X,\Z(2))_\tors$ 
are finite and their orders are the same (Lemma \ref{2ndcorr}), and 
(ii) under $\boldsymbol{TB}^{\boldsymbol{2}}(X)$ 
$H_\Wet^5(X,\Z(2))_\tors$ is finite and 
the order of $H_\et^5(X,\Z(2))_\tors$ is equal to that of $H_\Wet^5(X,\Z(2))_\tors$ (Lemma \ref{5th}). 

The following equality finishes our comparison:

\begin{theorem3}\label{secondmainthm}
Assume the same assumptions in $(\star \star)$ as in Theorem \ref{mainthm2} for $X$, and 
assume further $CH^2(X,1)\otimes \Q=0$ and that 
$H_\Wet^i(X,\Z(2))$ is finitely generated for all $i$. 
Then the following equality holds$:$
$$\frac{1}{R_1} \left| H_\ur^3(k(X),\Q/\Z(2)) \right|^2 \cdot \left| CH^2(X)_\tors \right|^2 
=\frac{1}{R_2} \left| H_\Wet^4(X,\Z(2))_\tors \right| \cdot \left| H_\Wet^6(X,\Z(2))_\tors \right|.$$
\end{theorem3}

This equality is obtained by comparing $R_1$ and $R_2$ directly. 
The method is similar to the proof of Theorem \ref{mainthm2}, 
and we use a certain Weil-\'etale cycle class map that we define in 
$\x$\ref{Weiletalecycle}. 
This equality is also considered as a comparison of higher Chow groups with Weil-\'etale motivic cohomology groups $H_\Wet^i(X,\Z(2))$ for $i=4$ and $i=6$.

\medskip
The organization of this paper is as follows. 
In the next section ($\x$\ref{pre}), we 
briefly review on conjectures in $(\star \star)$ and basic 
facts on \'etale cohomology groups. Moreover we prove Theorem \ref{conjecturetheorem} (b),(c). 
In $\x$\ref{sec:proof}, 
we prove our higher Chow formula (B) and Theorem \ref{conjecturetheorem} (a),(d). 
In $\x$\ref{sec:comparison}, 
we prove Theorem \ref{secondmainthm}. 

\medskip
{\bf Notations.} Throughout this paper, we use the following notation:
$$\begin{array}{ll}
p & :{\rm a\ prime\ number}, \\
\F_q & :{\rm the\ finite\ field\ consisting\ of\ }q{\rm \ elements\ with}\ q=p^a\ (a \in \Z,\ a>0),\\
\oF_q & :{\rm the\ algebraic\ closure\ of\ }\F_q,  \\
\Gamma:=\Gal(\oF_q/\F_q) & :{\rm the\ absolute\ Galois\ group\ of\ }\F_q, \\
X & :{\rm a\ smooth\ projective\ geometrically\ integral\ variety\ over\ }\F_q,\\
\overline{X}:=X \otimes_{\F_q} \oF_q. & { }
\end{array}$$
We often write $d$ for the dimension of $X$. 
For an abelian group $M$, $M^*$ denotes the Pontryagin dual 
$\Hom(M,\Q/\Z)$. 

\medskip
{\bf Acknowledgements}. 
I deeply express my gratitude to Prof.~Kanetomo Sato for numerous fruitful discussions and continual warmful encouragement. 
I would like to thank Proff.~Lars Hesselholt, Kazuhiro Fujiwara and Hiroshi Saito for 
carefully reading the earlier version of this paper and giving me many useful comments. 
I also thank Prof. Thomas Geisser for giving some valuable comments. 

\section{Preliminaries}\label{pre}

In this section, we give a brief review on some classical results and conjectures, 
and provide with some preliminary results. 
These results will play key roles in the later sections.

\subsection{The $\ell$-adic \'etale cohomology group $H_\et^i(X,\Z_\ell(n))$}\label{etale}

We recall the following important results on \'etale cohomology groups of a smooth projective 
variety $X$ over $\F_q$. 
Let $\mu_{m}$ be the \'etale sheaf of $m$-th roots of unity defined for a positive integer $m$. 
We define $\Z/\ell^\nu\Z(n)$ and $\Q_\ell/\Z_\ell(n)$ for a prime number $\ell$ and an integer $n$ as follows:
$$\Z/\ell^\nu\Z(n) := \left\{\begin{array}{cl}
\mu_{\ell^\nu}^{\otimes n} & (\ell \neq p) \\
W_\nu\Omega^n_{X,\log}[-n] & (\ell = p)
\end{array}\right.,\ \ 
\Q_\ell/\Z_\ell(n) := \lim_{\stackrel{\longrightarrow}{{\small \nu}}} \Z/\ell^\nu \Z(n).$$
where $W_\nu\Omega^n_{X,\log}$ denotes the \'etale subsheaf of the 
logarithmic part of the Hodge-Witt sheaf $W_\nu\Omega^n_X$ (\cite[pp.596 - 597, I.5.7]{Illusie}). 
For a positive integer $m=\ell_1^{e_1} \cdots \ell_r^{e_r}$ ($\ell_1,\cdots,\ell_r$ are prime numbers, and 
$e_1,\cdots,e_r \geq 1$ are integers) and an integer $n$, 
we define $$\Z/m\Z (n) := \bigoplus_{j=1}^r \Z/\ell_j^{e_j} \Z(n),\ \ \ 
\Q/\Z(n):= \lim_{\stackrel{\longrightarrow}{{\small m}}} \Z/m\Z(n).$$ 
With these notations, we define the $\ell$-adic \'etale cohomology groups $H_\et^i(X,\Z_\ell(n))$ for all primes $\ell$ 
and their product $H_\et^i(X,\hZ(n))$ as follows: 
$$H_\et^i(X,\Z_\ell(n)) :=  \lim_{\stackrel{\longleftarrow}{{\small \nu}}} H_\et^i(X,\Z/\ell^\nu \Z(n)),\ \ 
H_\et^i(X,\hZ(n)) := \prod_{{\rm all\ primes\ }\ell} H_\et^i(X,\Z_\ell(n)).$$
The first fundamental results concerning to these cohomology groups which are often used in this paper are as follows: 

\begin{proposition}\label{kaita}
Let $X$ be a smooth projective variety over $\F_q$.

\begin{enumerate}
\item[(1)] Let $\ell$ be a prime number different from $p$. Then, there is an isomorphism $H_\et^i(X,\Z_\ell(n)) \simeq H_\cont^i(X,\Z_\ell(n))$ for all 
$i$ and $n$, where $H_\cont^i(X,\Z_\ell(n))$ is the continuous \'etale cohomology 
group due to Jannsen (\cite{Jannsen}).

\item[(2)] There is a long exact sequence 
$$\cdots \to H_\et^i(X,\Z_\ell(n)) \to H_\et^i(X,\Q_\ell(n)) \to H_\et^i(X,\Q_\ell/\Z_\ell(n)) \to H_\et^{i+1}(X,\Z_\ell(n)) \to \cdots$$
for all prime numbers $\ell$.

\item[(3)] There is a long exact sequence 
$$\cdots \to H_\et^i(X,\hZ(n)) \to H_\et^i(X,\hZ(n))\otimes \Q \to H_\et^i(X,\Q/\Z(n)) \to H_\et^{i+1}(X,\hZ(n)) \to \cdots.$$

\end{enumerate}
\end{proposition}

\begin{proof}
(1) There is a following short exact sequence (\cite[(0.2)]{Jannsen}):
$$0 \to  {\lim_{\stackrel{\longleftarrow}{{\small \nu}}}}^{1} H_\et^{i-1}(X,\Z/\ell^\nu \Z(n)) \to 
H_\cont^i(X,\Z_\ell(n)) \to H_\et^i(X,\Z_\ell(n)) \to 0,$$
where $\displaystyle {\lim_{\stackrel{\longleftarrow}{{\small \nu}}}}^{1}$ denotes the derived limit. 
Since $H_\et^i(\overline{X},\Z/\ell^\nu \Z(n))$ is finite for $\overline{X} := X \otimes_{\F_q} \oF_q$ by \cite[Expos\'e XVI, Th\'eor\`eme 5.2]{SGA4}, 
we see $H_\et^i(X,\Z/\ell^\nu \Z(n))$ is also finite by the following exact sequence:
$$0 \to H_\et^{i-1}(\overline{X},\Z/\ell^\nu \Z(n))_\Gamma \to H_\et^i(X,\Z/\ell^\nu \Z(n)) 
\to H_\et^i(\overline{X},\Z/\ell^\nu \Z(n))^\Gamma \to 0,$$
which is obtained from the Hochschild-Serre spectral sequence 
$$E_2^{u,v}=H^u_{\rm Gal}(\Gamma,H_\et^v(\overline{X},\Z/\ell^\nu \Z(n))) \Rightarrow H_\et^{u+v}(X,\Z/\ell^\nu \Z(n))$$
($H_{\rm Gal}^i(\Gamma,-)$ denotes the Galois cohomology of $\Gamma = \Gal (\oF_q/\F_q)$) and the Pontryagin duality 
of the Galois cohomology of $\Gamma$. 
Therefore $\displaystyle {\lim_{\stackrel{\longleftarrow}{{\small \nu}}}}^{1} H_\et^{i-1}(X,\Z/\ell^\nu \Z(n))$ is $0$ for 
all $i$, and we have $H_\cont^i(X,\Z_\ell(n)) \simeq H_\et^i(X,\Z_\ell(n))$ by the first short 
exact sequence in this proof. 

\medskip
(2) We consider the following short exact sequence of complexes on $X_\et$ for positive integers $m$ and $m'$ 
(see \cite[p.779, Lemme 3]{CTSS} for the $p$-primary part):

\centerline{$(\dagger)$ \hfill $0 \to \Z/m\Z(n) \overset{\times m'}{\longrightarrow} \Z/mm'\Z(n) \to \Z/m'\Z(n) \to 0.$ \hfill { }}

\medskip
\noindent
When $(m,m')=(\ell^\nu,\ell^\mu)$ for a prime number $\ell$ ($\ell$ may be $p$), 
this exact sequence $(\dagger)$ yields a long exact sequence
$$\cdots \to H_\et^i(X,\Z/\ell^\nu \Z(n)) \to H_\et^i(X,\Z/\ell^{\nu+\mu}\Z(n)) \to H_\et^i(X,\Z/\ell^\mu \Z(n))\to \cdots.$$
If $\ell \neq p$, we have shown in (1) that these groups are all finite. If $\ell=p$, we see that these groups are all finite 
as well by \cite[Theorem 1.14]{Milne1}. Therefore in any case, we obtain the desired long exact sequence by taking 
the projective limit with respect to $\nu$ and then taking the inductive limit with respect to $\mu$. 

\medskip
(3) We have the following long exact sequence using the above sequence $(\dagger)$:
$$\cdots \to H_\et^i(X,\Z/m\Z(n)) \to H_\et^i(X,\Z/mm'\Z(n)) \to H_\et^i(X,\Z/m'\Z(n)) \to \cdots.$$
By taking projective limit with respect to $m$, and taking inductive limit with respect to $m'$, 
we obtain the desired long exact sequence. 
\end{proof}

\medskip
The second fundamental results are consequences of theorems by 
Deligne (\cite[Th\'eor\`eme 1.6]{Deligne1}) and Gabber (\cite{Gabber}). 
(cf. \cite[pp.779 - 786]{CTSS})

\begin{theorem}\label{DGconsequence}
Let $X$ be a smooth projective variety over $\F_q$. Then 

\begin{enumerate}
\item[(1)] 
The $\ell$-adic \'etale cohomology group $H_\et^i(X,\Z_\ell(n))$ is 
finitely generated for all $\ell,\ i$ and $n$. 

\item[(2)] 
For all $i\neq 2n,2n+1$, $H_\et^i(X,\hZ(n))$ is finite. 

\item[(3)] 
$H_\et^{2n}(X,\hZ(n))_\tors$ and $H_\et^{2n-1}(X,\Q/\Z(n))$ are both finite. 
\end{enumerate}
\end{theorem}

\begin{proof}
(1) See \cite[Lemma V.1.11]{Milne5}. 

(2) We prove the following two assertions:
\begin{enumerate}
\item[(a)] For all $i\neq 2n,2n+1$ and all primes $\ell$, 
$H_\et^i(X,\Z_\ell(n))$ is finite. 
\item[(b)] For all $i\neq 2n,2n+1$ and almost all primes $\ell \neq p$, 
$H_\et^i(X,\Z_\ell(n))=0$. 
\end{enumerate} 
We first prove (a) assuming $\ell \neq p$. 
By Proposition \ref{kaita} (2), 
there is a long exact sequence
$$\cdots \to H_\et^i(X,\Z_\ell(n)) \to H_\et^i(X,\Q_\ell(n)) \to H_\et^i(X,\Q_\ell/\Z_\ell(n)) \to H_\et^{i+1}(X,\Z_\ell(n)) \to \cdots.$$
We note that $H_\et^i(X,\Z_\ell(n))$ is a finitely generated $\Z_\ell$-module, 
$H_\et^i(X,\Q_\ell(n))$ is a finite-dimensional $\Q_\ell$-vector space, and 
$H_\et^i(X,\Q_\ell/\Z_\ell(n))$ is a cofinitely generated $\Z_\ell$-module. 
Using these properties, we see that the finiteness of $H_\et^i(X,\Q_\ell/\Z_\ell(n))$ is 
equivalent to that of $H_\et^i(X,\Z_\ell(n))$. 
Moreover, this finiteness condition is equivalent to 
$H_\et^i(X,\Q_\ell(n))=0$. 
Therefore, by the result of Colliot-Th\'el\`ene-Sansuc-Soul\'e (\cite[p.780, Th\'eor\`eme 2]{CTSS}), $H_\et^i(X,\Z_\ell(n))$ is finite for all $i \neq 2n,2n+1$. 
For the case $\ell=p$, we use the long exact sequence of Proposition \ref{kaita} (2) and 
the finiteness of $H_\et^i(X,\Q_p/\Z_p(n))$ for $i\neq 2n,2n+1$ (\cite[p.782, Th\'eor\`eme 3]{CTSS}). 

To prove (b), we consider the above long exact sequence of Proposition \ref{kaita} (2) again. 
By (a), $H_\et^i(X,\Z_\ell(n))$ is finite for all $i\neq 2n,2n+1$, and 
$H_\et^i(X,\Q_\ell(n))=0$ for all $i\neq 2n,2n+1$. 
Hence we see that $H_\et^i(X,\Z_\ell(n))=0$ for all $i\neq 2n,2n+1$ and $2n+2$. 
For the case $i=2n+2$, we need to prove that 
$H_\et^{i-1}(X,\Q_\ell/\Z_\ell(n))$ is divisible for almost all $\ell \neq p$, which one can check by repeating 
the arguments in the proof of \cite[p.780, Th\'eor\`eme 2]{CTSS}  due to 
Colliot-Th\'el\`ene-Sansuc-Soul\'e. 

(3) We use the long exact sequence of Proposition \ref{kaita} (3). 
By \cite[p.780, Th\'eor\`eme 2]{CTSS}, 
$H_\et^{2n-1}(X,\Q/\Z(n))$ is finite. 
We have already seen that $H_\et^{2n-1}(X,\hZ(n))\otimes \Q=0$ by the above (2). 
Therefore, using the above long exact sequence, we see $H_\et^{2n}(X,\hZ(n))_\tors$ is finite. 
\end{proof}

We will study $H_\et^5(X,\hZ(2))_\tors$ in $\x$\ref{finitenessresult}, which is a special case of $H_\et^{2n+1}(X,\hZ(n))_\tors$ that was not treated in the above 
theorem.

The third important result is the following duality theorem 
due to Milne (\cite[Theorem 1.14]{Milne1}):

\begin{theorem}[Milne]\label{Milneduality} 
Let $X$ be a $4$-dimensional smooth projective variety over $\F_q$. 
Then there is a canonical non-degenerate pairing of finite groups
$$H_\et^i(X,\Z/\ell^\nu\Z(2)) \times H_\et^{9-i}(X,\Z/\ell^\nu\Z(2)) \to \Z/\ell^\nu\Z$$
for all primes $\ell$. 
Consequently, taking limits, we have the following Pontryagin duality between compact groups and discrete groups$:$
$$H_\et^i(X,\hZ(2)) \times H_\et^{9-i}(X,\Q/\Z(2)) \to \Q/\Z.$$
\end{theorem}

\noindent
We often use this duality theorem in this paper.

\subsection{Classical deep conjectures}\label{conjecture}

We remind the basic notations and definitions of equivalence relations on algebraic cycles 
on a smooth projective variety $X$ defined over $\F_q$ (cf. \cite{Kleiman}). 
Let $Z^j(X)$ be the group of (algebraic) cycles of codimension $j$ on $X$, and 
$Z^j_\sim(X)$ be the group of cycles of codimension $j$ 
which are equivalent to $0$ with respect to an adequate 
equivalence relation $\sim$. 
In this paper, we consider three equivalence relations; 
rational equivalence ($\rat$), $\ell$-adic homological equivalence for a prime number $\ell$ ($\hom(\ell)$), and 
numerical equivalence ($\num$). Let $\overline{X} := X\otimes_{\F_q} \oF_q$. 
Recall that a cycle $C$ of codimension $j$ is said to be 
$\ell$-adic homologically equivalent to $0$ 
if $C$ goes to $0$ under the $\ell$-adic cycle class map:  
$$Z^j(X)\otimes \Q_\ell \to H_\et^{2j}(\overline{X},\Q_\ell(j))\ \ (\ell \neq p),\ \ \ Z^j(X) \otimes \Q_p \to H_\crys^{2j}(X/W) \otimes_W K\ \ (\ell = p),$$
where $H_\crys^{2j}(X/W)$ denotes the crystalline cohomology group of $X$ over $W$, 
$W:=W(\F_q)$ denotes the ring of Witt vectors on $\F_q$, and $K$ denotes the field of fractions of $W$. 
The relationship between these three equivalence relations is as follows:
$$Z_\rat^j(X) \subset Z_{\hom(\ell)}^j(X) \subset Z_\num^j(X).$$
We define the Chow group of codimension $j$ as $CH^j(X) := Z^j(X)/Z_\rat^j(X)$. 
Tate and Beilinson raised the following conjectures for a prime number $\ell$ 
different from $p$:

\begin{enumerate}
\item ${\boldsymbol{T}^{\boldsymbol{n}}(X,\ell)}$: The $\ell$-adic cycle class map 
$$CH^n(X)\otimes \Q_\ell \to H^{2n}_\et(\oX,\Q_\ell(n))^{\Gamma}$$
is surjective, where $\Gamma$ denotes $\Gal(\oF_q/\F_q)$. 
\item ${\boldsymbol{SS}^{\boldsymbol{n}}(X,\ell)}$: 
$\Gamma$ acts semisimply on the generalized eigenspace of eigenvalue $1$ in $H_\et^{2n}(\overline{X},\Q_\ell(n))$. 
\item ${\boldsymbol{E}^{\boldsymbol{n}}(X,\ell)}$: With rational coefficients, 
the $\ell$-adic homological equivalence agrees with the numerical equivalence ({\em i.e.}, $Z_{\hom(\ell)}^n(X)\otimes \Q = Z_\num^n(X)\otimes \Q$).  
\item ${\boldsymbol{B}^{\boldsymbol{n}}(X,\ell)}$: With rational coefficients, the rational equivalence agrees with the $\ell$-adic homological equivalence ({\em i.e.}, $Z_\rat^n(X)\otimes \Q=Z_{\hom(\ell)}^n(X)\otimes \Q$). 
\end{enumerate}

\noindent
For $\ell = p$, we need to introduce the following $p$-adic counterparts of these conjectures. 

\begin{enumerate}
\item ${\boldsymbol{T}^{\boldsymbol{n}}(X,p)}$: Let 
$\varphi:H_\crys^{2n}(X/W) \to H_\crys^{2n}(X/W)$ be the crystalline Frobenius 
induced by the absolute Frobenius endomorphism on $X$. 
Then the $p$-adic cycle class map 
$$CH^n(X)\otimes \Q_p \to \left( H^{2n}_\crys(X/W) \otimes_W K \right)^{\varphi = p^n}$$
is surjective. 
\item ${\boldsymbol{SS}^{\boldsymbol{n}}(X,p)}$: 
The crystalline Frobenius $\varphi$ acts semisimply on the generalized eigenspace of eigenvalue $p^n$ in $H_\crys^{2n}(X/W) \otimes_W K$. 
\item ${\boldsymbol{E}^{\boldsymbol{n}}(X,p)}$: With rational coefficients, 
the $p$-adic homological equivalence agrees with the numerical equivalence ({\em i.e.}, $Z_{\hom(p)}^n(X)\otimes \Q = Z_\num^n(X)\otimes \Q$).  
\item ${\boldsymbol{B}^{\boldsymbol{n}}(X,p)}$: With rational coefficients, the rational equivalence agrees with the $p$-adic homological equivalence ({\em i.e.}, $Z_\rat^n(X)\otimes \Q=Z_{\hom(p)}^n(X)\otimes \Q$). 
\end{enumerate}

In this paper, we sometimes assume the following {\em Tate \& Beilinson conjecture} for $n=2$:

\medskip
\begin{quotation}
${\boldsymbol{TB}^{\boldsymbol{n}}(X)}$: {\em $CH^n(X)\otimes \Q$ is finite\text{-}dimensional over $\Q$, and 
the order $\rho_n$ of the pole of $\zeta(X,s)$ at $s=n$ is equal to $\dim_\Q \left( CH^n(X) \otimes \Q \right)$. }
\end{quotation}

\medskip
This conjecture is sometimes called strong Tate and Beilinson conjecture. 
The original form of this conjecture (cf. \cite[Theorem 2.9 (e)]{Tate2}) 
asserts that the rank of the group of numerical equivalence classes ({\em not} rational equivalence classes) of cycles of 
codimension $n$ on $X$ is equal to the order $\rho_n$ defined in the above, and 
Tate proved that it is equivalent to ${\boldsymbol{T}^{\boldsymbol{n}}(X,\ell)}+{\boldsymbol{SS}^{\boldsymbol{n}}(X,\ell)}$ for a single $\ell$. 

\begin{proposition}\label{TBconjecture}
Let $X$ be a $d$-dimensional smooth projective variety over $\F_q$. Then 
the Tate \& Beilinson conjecture ${\boldsymbol{TB}^{\boldsymbol{n}}(X)}$ is equivalent to the combination of the conjectures $\boldsymbol{T}^{\boldsymbol{n}}(X,\ell)$, $\boldsymbol{SS}^{\boldsymbol{n}}(X,\ell)$, 
$\boldsymbol{E}^{\boldsymbol{n}}(X,\ell)$ and $\boldsymbol{B}^{\boldsymbol{n}}(X,\ell)$ for all prime numbers $\ell$. 
\end{proposition}

\begin{proof}
Obviously, the condition 
$\dim_\Q CH^n(X)\otimes \Q=\dim_\Q A^n_\num(X)\otimes \Q$ is equivalent to 
the condition $Z_\rat^n(X)\otimes \Q=Z_\num^n(X) \otimes \Q$. 
Hence the assertion for the $\ell \neq p$ follows from a result of Tate (\cite[Theorem 2.9]{Tate2}). 
As for the case $\ell=p$, 
one can obtain analogous results as in \cite[Theorem 2.9]{Tate2} by 
replacing $\ell$-adic \'etale cohomology with crystalline cohomology (cf.~\cite[Theorem 1.11]{Milne10}). 
Thus we obtain the proposition. 
\end{proof}

$\boldsymbol{TB}^{\boldsymbol{2}}(X)$ is a very strong assumption. 
As consequences of $\boldsymbol{TB}^{\boldsymbol{2}}(X)$, 
we have the following facts which are useful in this paper. 

\begin{proposition}\label{TBconsequence}
Let $X$ be a $4$-dimensional smooth projective geometrically integral variety over $\F_q$. Then 

\begin{enumerate}
\item[(1)] 
$\boldsymbol{TB}^{\boldsymbol{2}}(X)$ implies that the $\ell$-adic cycle class map
$$CH^2(X) \otimes \Q_\ell \to H_\et^4(X,\Q_\ell(2))$$
is bijective for all prime numbers $\ell$.

\item[(2)]
Assume $\boldsymbol{TB}^{\boldsymbol{2}}(X)$ and that $CH^2(X)$ is finitely generated. Then the $\ell$-adic 
cycle class map with $\Z_\ell$-coefficient 
$$\r_{\Z_\ell}^2:CH^2(X)\otimes \Z_\ell \to H_\et^4(X,\Z_\ell(2))$$
is injective for all prime numbers $\ell$. 
\end{enumerate}
\end{proposition}

\begin{proof}
(1) We first prove the case $\ell \neq p$. 
By a result due to Tate (\cite[Theorem 2.9]{Tate2}), 
$\boldsymbol{TB}^{\boldsymbol{2}}(X)$ implies the following 
$\ell$-adic cycle class map is bijective:
$$CH^2(X) \otimes \Q_\ell \overset{\simeq}{\longrightarrow} H_\et^4(\overline{X},\Q_\ell(2))^\Gamma.$$
Thus it remains to prove an isomorphism 
$H_\et^4(X,\Q_\ell(2)) \overset{\simeq}{\longrightarrow} H_\et^4(\overline{X},\Q_\ell(2))^\Gamma.$ 
Since $\Gamma$ has cohomological dimension one, 
the Hochschild-Serre spectral sequence for the continuous \'etale cohomology groups (\cite[(3.5) b)]{Jannsen})
$$E_2^{u,v}=H^u_\cont(\Gamma,H^v_\et(\overline{X},\Q_\ell(2))) \Rightarrow H^{u+v}_\cont(X,\Q_\ell(2))$$
and the isomorphism $H_\cont^i(X,\Q_\ell(2)) \simeq H_\et^i(X,\Q_\ell(2))$ in Proposition \ref{kaita} (1) give the following exact sequence:

\medskip
\centerline{$(\blacklozenge)$ \hfill $0 \to H^{3}_\et(\overline{X},\Q_\ell(2))_\Gamma \to H_\et^{4}(X,\Q_\ell(2)) \to H_\et^4(\overline{X},\Q_\ell(2))^\Gamma \to 0$.\hfill { }}

\medskip
\noindent
On the other hand, we have the following exact sequence for $M:=H_\et^3(\overline{X},\Q_\ell(2))$:
$$0 \to M^\Gamma \to 
M \overset{1-\Fr}{\longrightarrow} 
M \to M_\Gamma \to 0.$$
Here $\Fr$ denotes the canonical generator of $\Gamma$ acting on $M$. 
Since $M \overset{1-\Fr}{\longrightarrow} M$ is an isomorphism 
by the Weil conjecture (Deligne's theorem \cite[Th\'eor\`eme 1.6]{Deligne1}), 
we see that $M_\Gamma =0$. 
Therefore 
we have $H_\et^4(X,\Q_\ell(2)) \overset{\simeq}{\longrightarrow} H_\et^4(\overline{X},\Q_\ell(2))^\Gamma$ by $(\blacklozenge)$, which proves the assertion if $\ell \neq p$. 

We prove the case $\ell=p$. 
By Proposition \ref{TBconjecture}, 
$\boldsymbol{TB}^{\boldsymbol{2}}(X)$ implies that the following $p$-adic cycle class map is bijective:
$$CH^2(X) \otimes \Q_p \overset{\simeq}{\longrightarrow} \left( H_\crys^4 (X/W) \otimes_W K \right)^{\varphi = p^2}.$$
We consider the long exact sequence (cf. \cite[p.785, (37)]{CTSS}, \cite[I.5.7.2]{Illusie}) 
$$\cdots \to H_\crys^3(X/W) \otimes_W K \overset{p^2-\varphi}{\longrightarrow} H_\crys^3(X/W)\otimes_W K \to $$
$$\ \ \ H_\et^4(X,\Q_p(2)) \overset{\tau}{\longrightarrow} H_\crys^4(X/W)\otimes_W K \overset{p^2 - \varphi}{\longrightarrow} 
H_\crys^4(X/W)\otimes_W K \to \cdots,\ \ \ \ \ \ \ \ \ \ \ \ \ \ $$
where we have used the compatibility between the crystalline Frobenius operator and the Frobenius operator acting on 
the de Rham-Witt complex (\cite[p.603, II (1.2.3)]{Illusie}). 
From this sequence, we see the map 
$$\tau:H^4_\et(X,\Q_p(2)) \to \left( H_\crys^4(X/W) \otimes_W K \right)^{\varphi = p^2}$$
is surjective. 
It remains to prove that the map $p^2-\varphi:H_\crys^3(X/W)\otimes_W K \to H_\crys^3(X/W) \otimes_W K$ is surjective. 
Let $a$ be the integer with $q=p^a$. Then, $\Phi := \varphi^a$ is $K$-linear and we have 
$$q^2-\Phi= \left( p^2-\varphi \right) \left(p^{2(a-1)}+p^{2(a-2)}\varphi+\cdots+p^2\varphi^{a-2} + \varphi^{a-1} \right)$$
as endomorphism on $H_\crys^3(X/W)\otimes_W K$. Moreover, 
$q^2-\Phi$ is bijective on $H_\crys^3(X/W) \otimes_W K$ by the Katz-Messing theorem (\cite{KatzMessing}). 
Hence $p^2-\varphi$ is surjective on $H_\crys^3(X/W)\otimes_W K$ and we obtain the assertion. 

\medskip
(2) By (1), it remains to prove that $\r_{\Z_\ell}^2$ is injective on torsion 
for all primes $\ell$, which has been proved by Colliot-Th\'el\`ene-Sansuc-Soul\'e (\cite[p.787, Th\'eor\`eme 4]{CTSS} for the case of $\ell \neq p$) and Gros (\cite[Th\'eor\`eme 2.1.11]{Gros} for the case of $\ell=p$). 
\end{proof}

\subsection{Finiteness of $H_\ur^3(k(X),\Q_\ell/\Z_\ell(2))$ and $H_\et^5(X,\Z(2))$}\label{finitenessresult}

In our second main theorem (Theorem \ref{mainthm2}), 
we assume three conjectures in $(\star \star)$, 
but these conjectures are not independent of each other. 
In this subsection, we prove Theorem \ref{finiteness} below concerning their relations. 
We recall the cycle complex of $X$ due to Bloch (\cite{Bloch2}), 
which is a cochain complex of abelian groups of the following form: 
$$\cdots \to z^n(X,j) \overset{d_j}{\longrightarrow} z^n(X,j-1) 
\overset{d_{j-1}}{\longrightarrow} \cdots \overset{d_1}{\longrightarrow} 
z^n(X,0).$$
Here $n$ is an integer, and $z^n(X,j)$ is placed in degree $-j$. 
Each $z^n(X,j)$ is defined as the free abelian group 
generated by all codimension $n$ integral closed subvarieties 
on $X \times \Delta^j$ of codimension $n$ 
which intersect properly with all faces of $X\times \Delta^j$, 
where $\Delta^\bullet$ is the standard cosimplicial scheme over $\Spec~\Z$ defined as follows:
$$\Delta^j := \Spec~\Z[x_0,x_1,\cdots,x_j]/(x_0+x_1+\cdots+x_j-1).$$
We define $\Z(n)$ as the \'etale sheafification of the presheaf of 
this cycle complex shifted by degree $-2n$. 
Later we use $\Q(n)$ that is defined as $\Q(n) := \Z(n) \otimes \Q$.

\begin{theorem}\label{finiteness}
Let $X$ be a $4$-dimensional smooth projective geometrically integral variety 
defined over $\F_q$.

\begin{enumerate}
\item[(a)] 
If ${\boldsymbol{TB}^{\boldsymbol{2}}(X)}$ holds and $CH^2(X)$ is finitely generated, then $H_\ur^3(k(X),\Q_\ell/\Z_\ell(2))$ is finite for all prime numbers $\ell$. 

\item[(b)] ${\boldsymbol{TB}^{\boldsymbol{2}}(X)}$ implies that $H_\et^{5}(X,\Z(2))$ is finite.

\item[(c)] ${\boldsymbol{TB}^{\boldsymbol{2}}(X)}$ implies that the intersection pairing
$$CH^2(X) \times CH^2(X) \to CH^4(X)\simeq CH_0(X) \overset{\deg}{\longrightarrow} \Z$$ is non-degenerate when tensored with $\Q$. 
\end{enumerate}
\end{theorem}

Theorem \ref{finiteness} (a) will be used later in the proof of Theorem 1 (a) (see $\x$\ref{subsec:ell} below). 
Theorem \ref{conjecturetheorem} (b),(c) follows from the corresponding assertions of this theorem.

\begin{proof} 
(a) We fix a prime number $\ell$, and consider the following local-global spectral sequence (cf. \cite{BlochOgus}):
$$E_1^{u,v}=\bigoplus_{x\in X^{(u)}} H_{x,\et}^{u+v}(X,\Q_\ell/\Z_\ell(2)) \Rightarrow H_\et^{u+v}(X,\Q_\ell/\Z_\ell(2)).$$
By the relative cohomological purity theorem (cf. \cite[Theorem VI.5.1]{Milne5}), 
we have the following long exact sequence:
$$\cdots \to H_\et^3(X,\Q_\ell/\Z_\ell(2)) \to H_\ur^3(k(X),\Q_\ell/\Z_\ell(2)) \to 
CH^2(X)\otimes \Q_\ell/\Z_\ell \overset{\r^2_{\Q_\ell/\Z_\ell}}{\longrightarrow} H_\et^4(X,\Q_\ell/\Z_\ell(2)) \to \cdots.$$
To prove the finiteness of $H_\ur^3(k(X),\Q_\ell/\Z_\ell(2))$, 
we need to show that of $\Ker(\r^2_{\Q_\ell/\Z_\ell})$. 
By the assumptions and Proposition \ref{TBconsequence}, 
the cycle class map 
$\r_{\Z_\ell}^2:CH^2(X)\otimes \Z_\ell \to H_\et^4(X,\Z_\ell(2))$ is injective with finite cokernel. 
Based on these facts, we consider the following diagram with exact rows and 
columns:
{\small 
$$\xymatrix{
{ } & { } & { } & 0 \ar[d] \ar[r] & \Ker(\r_{\Q_\ell/\Z_\ell}^2) \ar[d] & { } \\
0 \ar[r] & CH^2(X)_{\ell\text{-}\tors} \ar[d]^\r \ar[r] & CH^2(X)\otimes \Z_\ell \ar[d]^{\varrho_{\Z_\ell}^2} \ar[r] & CH^2(X) \otimes \Q_\ell \ar[d]^{\r_{\Q_\ell}^2}_{\simeq} \ar[r] & CH^2(X) \otimes \Q_\ell/\Z_\ell \ar[d]^{\r_{\Q_\ell/\Z_\ell}^2}  \ar[r] & 0 \\
0 \ar[r] & H_\et^3(X,\Q_\ell/\Z_\ell(2)) \ar[r]^{\ \ (*)} \ar[d] & H_\et^4(X,\Z_\ell(2)) \ar[r] \ar[d] & H_\et^4(X,\Q_\ell(2)) \ar[r] \ar[d] & H_\et^4(X,\Q_\ell/\Z_\ell(2)) & { } \\ 
0 \ar[r] & \Coker(\r) \ar[r] & { } \Coker(\r_{\Z_\ell}^2) \ar[r] & 0. & { }
}$$}

\noindent
The arrow $(*)$ is injective because $H_\et^3(X,\Q_\ell(2))=0$ by the arguments 
in the proof of Theorem \ref{DGconsequence} (2). 
Since $H_\et^3(X,\Q_\ell/\Z_\ell(2))$ is finite by Theorem \ref{DGconsequence} (3), $\Coker(\r)$ is also finite, 
and the map $H_\et^4(X,\Z_\ell(2)) \to H_\et^4(X,\Q_\ell(2))$ has finite kernel. 
By a diagram chase, we see $\Ker(\r^2_{\Q_\ell/\Z_\ell})$ is finite. 
As a result, $H_\ur^3(k(X),\Q_\ell/\Z_\ell(2))$ is finite.

\medskip
(b) By a result of Milne (\cite[p.95]{Milne2}), 
${\boldsymbol{TB}^{\boldsymbol{2}}(X)}$ yields 
an isomorphism $H_\et^5(X,\Z(2)) \overset{\simeq}{\longrightarrow} 
H_\et^5(X,\hZ(2))_\tors$. 
Hence, we need to prove the finiteness of $H_\et^5(X,\hZ(2))_\tors$. 
On the other hand, ${\boldsymbol{TB}^{\boldsymbol{2}}(X)}$ implies 
Tate's partial 
semisimplicity conjecture ${\boldsymbol{SS}^{\boldsymbol{2}}(X,\ell)}$ for all $\ell$ 
(Theorem \ref{TBconjecture}), and 
${\boldsymbol{SS}^{\boldsymbol{2}}(X,\ell)}$ for all $\ell$ implies the finiteness of $H_\et^5(X,\hZ(2))_\tors$ 
by another result of Milne (\cite[Introduction \& Proposition 6.6]{Milne2}). 
This finishes our proof.

\medskip
(c) Fix a prime number $\ell \neq p$. Then we have 
$$CH^2(X) \otimes \Q_\ell \simeq H_\et^4(\overline{X},\Q_\ell(2))^\Gamma \simeq 
H_\et^4(\overline{X},\Q_\ell(2))_\Gamma$$
by Proposition \ref{TBconsequence} (1) and $\boldsymbol{SS}^{\boldsymbol{2}}(X,\ell)$. 
Hence the assertion follows from the Poincar\'e duality and Theorem \ref{TBconjecture}. 
\end{proof}

\section{Proof of the higher Chow formula (Theorem \ref{mainthm2})}\label{sec:proof}

In this section, we prove Theorem \ref{conjecturetheorem} (a),(d) and 
higher Chow formula (B) stated in Theorem \ref{mainthm2}.

\subsection{Proof of Theorem \ref{conjecturetheorem} (d)}\label{subsec:rewrite}

We would like to rewrite \'etale cohomology groups $H_\et^i(X,\hZ(2))$ for $i\neq 4,5$ in terms of 
higher Chow groups. 
These are calculated as follows: 

\begin{lemma}\label{tech}
We have the following isomorphisms$:$ 
$$H^0_\et(X,\hZ(2))=0,\ H^i_\et(X,\hZ(2)) \simeq H^{i-1}_\et(X,\Q/\Z(2)) \simeq CH^2(X,4-i)_\tors\ \ \ (i=1,2,3),$$
$$(H^i_\et(X,\hZ(2)))^*\simeq H^{9-i}_\et(X,\Q/\Z(2)) \simeq CH^2(X,i-6)_\tors\ \ (i=7,8,9),\ \ H_\et^i(X,\hZ(2))=0\ \ (i\geq 10),$$
where $*$ means the Pontryagin dual. Moreover $CH^2(X,i)_\tors=0$ for $i\geq 4$. 
\end{lemma}

This proves Theorem \ref{conjecturetheorem} (d) by the result of Theorem \ref{DGconsequence} (2). 

In the proof of this lemma, 
we need the following facts and Theorem \ref{Milneduality}:

\begin{enumerate}
\item[(1)] We have the isomorphisms
$$H_\et^i(X,\hZ(2)) \simeq H_\et^{i-1}(X,\Q/\Z(2))\ \ \ {\rm for\ } i\neq 4,5,6.$$
This follows from the following long exact sequence:
$$\cdots \to H^i_\et(X,\hZ(2)) \to H^i_\et(X,\hZ(2))\otimes \Q \to H^i_\et(X,\Q/\Z(2)) 
\to H^{i+1}_\et(X,\hZ(2)) \to \cdots,$$
which is given in Proposition \ref{kaita} (3). 
\item[(2)] 
We have the isomorphisms
$$H_\et^{i-1}(X,\Q/\Z(2)) \simeq CH^2(X,4-i)_\tors\ \ {\rm for\ }i\leq 3.$$
This follows from the following exact sequence due to Geisser-Levine 
(Geisser-Levine \cite[Theorem 1.5]{GeisserLevine1},\cite{GeisserLevine2}) with Merkur'ev-Suslin's theorem (\cite{MS}) in our case:
$$0 \to CH^2(X,4-i)\otimes \Q/\Z \to H_\et^{i}(X,\Q/\Z(2)) \to CH^2(X,3-i)_\tors \to 0\ \ \ (i\leq 3),$$
and the fact that $CH^2(X,4-i)\otimes \Q/\Z$ is finite and divisible at the same time. 
\end{enumerate}

\begin{proof}[Proof of Lemma \ref{tech}]
The assertions of Lemma \ref{tech} follow from the above two facts (1),(2) and Theorem \ref{Milneduality}: 
$$H^0_\et(X,\hZ(2)) \underset{(1)}{\simeq} H^{-1}_\et(X,\Q/\Z(2))=0,$$
$$H^i_\et(X,\hZ(2)) \underset{(1)}{\simeq} H^{i-1}_\et(X,\Q/\Z(2)) \underset{(2)}{\simeq} CH^2(X,4-i)_\tors\ \ (i=1,2,3),$$
$$\left( H^{i}_\et(X,\hZ(2)) \right)^* \underset{{\rm Theorem}\ \ref{Milneduality}}{\simeq} H^{9-i}_\et(X,\Q/\Z(2)) \underset{(2)}{\simeq} CH^2(X,i-6)_\tors\ \ (i=7,8,9).$$
For a negative integer $i$, $H_\et^i(X,\Q/\Z(2))=0$. Therefore we see $CH^2(X,i)_\tors=0$ for $i\geq 4$ by the 
above fact (2). This completes the proof. 
\end{proof}

\subsection{Proof of Theorem \ref{conjecturetheorem} (a)}\label{subsec:ell}
By Theorem \ref{finiteness} (a), we know that $H_\ur^3(k(X),\Q_\ell/\Z_\ell(2))$ is finite for all prime numbers $\ell$. 
Thus the remaining task is to prove that it is $0$ for almost all $\ell$. 
As we have already discussed in the introduction, Milne, Kahn and Chambert-Loir showed the following 
general formula on $\zeta^*(X,2)$: 

\medskip
\noindent
$(\natural)$ \hfill 
$\displaystyle \zeta^*(X,2)=(-1)^{S(2)} \cdot q^{\chi(X,\O_X,2)} \cdot \underbrace{\prod_{i\neq 4,5} 
|H_\et^i(X,\hZ(2))|^{(-1)^i}}_{(\spadesuit)} \cdot \underbrace{\frac{|\Ker (\e^4)|}{|\Coker (\e^4)|}}_{(\clubsuit)},$ 
\hfill { } \\

\medskip
\noindent
where $\e^4 : H^4_\et(X,\hZ(2)) \to H^5_\et(X,\hZ(2))$ is the cup product with 
the canonical element $1\in \hZ \simeq H_\et^1(\Spec~\F_q,\hZ)$ explained in the introduction. 
Let $\ell$ be a prime number. We know that the $\ell$-adic part of $(\spadesuit)$ is $1$ for almost all $\ell$ by 
Theorem \ref{DGconsequence} (2). Therefore, we see from the above equality of rational numbers for $\zeta^*(X,2)$ that the $\ell$-adic part of $(\clubsuit)$ is $1$ for almost all $\ell$.  
In the below, we will relate the $\ell$-adic part of $(\clubsuit)$ to the unramified cohomology group 
$H_\ur^3(k(X),\Q_\ell/\Z_\ell(2))$ for all prime numbers $\ell$. 

We use a similar argument as in \cite[Appendix B]{SaitoSato}. We consider the following diagram:

$$\xyoption{curve}
\xymatrix{
CH^2(X)\otimes \Z_\ell \ar[rrr]^{\Theta_\ell\ \ \ \ \ \ \ \ \ \ } \ar[dd]_{\alpha_\ell} & & & \Hom_{\Z_\ell\text{-}\mod}(CH^2(X)\otimes \Z_\ell,\Z_\ell)\\
{ } & & & \Hom_{\Z_\ell\text{-}\mod}(H_\et^4(X,\Z_\ell(2)),\Z_\ell) \ar[u]_{\beta_{2,\ell}} \\
H^4_\et(X,\Z_\ell(2)) \ar[rrr]_{\e^4_\ell} & & & H^5_\et(X,\Z_\ell(2)) \ar@{->>}[u]_{\beta_{1,\ell}} \ar@/_6pc/[uu]_{\beta_\ell}
.}$$
Here $\e_\ell^4:H_\et^4(X,\Z_\ell(2)) \to H_\et^5(X,\Z_\ell(2))$ is the cup product with canonical element $1\in \Z_\ell \simeq H_\et^1(\Spec~\F_q,\Z_\ell)$, $\alpha_\ell:CH^2(X)\otimes \Z_\ell \to H^4_\et(X,\Z_\ell(2))$ is the codimension two cycle class map, 
$\beta_{1,\ell}$ is the map induced by the following pairing
$$H_\et^4(X,\Z_\ell(2)) \times H^5_\et(X,\Z_\ell(2)) \to H^9_\et(X,\Z_\ell(2)) \simeq \Z_\ell$$
due to Milne (\cite[Lemma 5.3]{Milne2}), and $\beta_{2,\ell}$ is the induced map 
by the cycle class map $\alpha_\ell$. We denote $\beta_\ell$ as a composition $\beta_{2,\ell} \circ \beta_{1,\ell}$. 
By Milne's result (\cite[Lemma 5.4]{Milne2}),  
this diagram commutes. We have the following facts:

\medskip
\begin{enumerate}
\item Assuming $\boldsymbol{TB}^{\boldsymbol{2}}(X)$, 
$\alpha_\ell$ is injective by Theorem \ref{TBconsequence} (2). 
\item Assuming $\boldsymbol{TB}^{\boldsymbol{2}}(X)$ and 
the finite generation of $CH^2(X)$, 
we have the finiteness of $\Coker (\alpha_\ell)$ and $|\Coker(\alpha_\ell)|=|H_\ur^3(k(X),\Q_\ell/\Z_\ell(2))|$  by Theorem \ref{finiteness} (a) (\cite[Proposition 4.3.5]{SaitoSato}).  
\item $\beta_{1,\ell}$ is surjective and $\Ker (\beta_{1,\ell})=H^5_\et(X,\Z_\ell(2))_\tors$ by a result of Milne (\cite[Lemma 5.3]{Milne2}). We remark that $H_\et^5(X,\Z_\ell(2))_\tors$ is the $\ell$-primary torsion part of $H_\et^5(X,\Z(2))$ because 
there is an isomorphism $H_\et^5(X,\Z(2)) \overset{\simeq}{\longrightarrow} H_\et^5(X,\hZ(2))_\tors$ 
under $\boldsymbol{TB}^{\boldsymbol{2}}(X)$ by Milne (\cite[p.95]{Milne2}).
\item Since we assume $\boldsymbol{TB}^{\boldsymbol{2}}(X)$ and the finite generation of $CH^2(X)$, $\beta_{2,\ell}$ is injective by the finiteness of $\Coker(\alpha_\ell)$ and Theorem \ref{conjecturetheorem} (a), 
and $\Ker(\beta_\ell)=\Ker(\beta_{1,\ell})$.  
\item By Theorem \ref{TBconsequence} (2), $\Ker(\Theta_\ell)=CH^2(X)_{\ell\text{-}\tors}$ and $|\Coker(\Theta_\ell)|=\left< R_1 \right>_\ell$. Here 
$CH^2(X)_{\ell\text{-}\tors}$ denotes the $\ell$-primary torsion part of $CH^2(X)$, and 
we denote $\left< N \right>_\ell:=\ell^{\ord_\ell(N)}$ for a natural number $N$. 
\end{enumerate}
We prove the following fact concerning the cokernel of $\beta_\ell$:

\begin{lemma}\label{sublemma1} Assume $\boldsymbol{TB}^{\boldsymbol{2}}(X)$ and 
that $CH^2(X)$ is finitely generated. Then 
we have the following equality for all $\ell$$:$
$$|\Coker(\beta_\ell)|=|\Coker(\beta_{2,\ell})|=\frac{|H_\ur^3(k(X),\Q_\ell/\Z_\ell(2))| \cdot | CH^2(X)_{\ell\text{-}\tors}|}{|H^4_\et(X,\Z_\ell(2))_\tors|}.$$
\end{lemma}

\noindent
We first finish the proof of Theorem 1 (a) admitting this lemma. 
By the above facts 1--5 and the assumptions in $(\star \star)$ and Theorem \ref{finiteness}, 
the maps $\Theta_\ell,\ \alpha_\ell,\ \beta_\ell$ and $\e^4_\ell$ have 
finite kernel and cokernel. Hence, we have the following equality of rational numbers: 

\medskip
\noindent
$(\flat)$ \hfill 
$\displaystyle \frac{|\Ker (\Theta_\ell)|}{|\Coker (\Theta_\ell)|}=
\frac{|\Ker(\alpha_\ell)|}{|\Coker(\alpha_\ell)|} \cdot \frac{|\Ker(\e^4_\ell)|}{|\Coker(\e^4_\ell)|} \cdot 
\frac{|\Ker(\beta_\ell)|}{|\Coker(\beta_\ell)|}.$ \hfill { } \\

\noindent
By Lemma \ref{sublemma1} and the above facts 1--5, 
we rewrite this equality $(\flat)$ as follows:
$$|H_\ur^3(k(X),\Q_\ell/\Z_\ell(2))|^2 = 
\underbrace{\frac{|\Ker(\e_\ell^4)|}{|\Coker(\e_\ell^4)|}}_{\ell\text{-}{\rm adic\ part\ of\ }(\clubsuit)} \!\!\! \cdot \ 
\frac{\left< R_1 \right>_\ell \cdot |H_\et^4(X,\Z_\ell(2))_\tors| \cdot |H_\et^5(X,\Z_\ell(2))_\tors|}{|CH^2(X)_{\ell\text{-}\tors}|^2}.$$
Since we assumed $\boldsymbol{TB}^{\boldsymbol{2}}(X)$ and the finite generation of $CH^2(X)$, both 
$\left< R_1 \right>_\ell$ and $|CH^2(X)_{\ell\text{-}\tors}|$ are $1$ for almost all $\ell$.  
$|H_\et^5(X,\Z_\ell(2))_\tors|$ is $1$ for almost all $\ell$ under $\boldsymbol{TB}^{\boldsymbol{2}}(X)$ by the above fact 3 and Theorem \ref{finiteness} (b). 
$|H_\et^4(X,\Z_\ell(2))_\tors|$ is also $1$ for almost all $\ell$ by Theorem \ref{DGconsequence} (3). 
Therefore, $|H_\ur^3(k(X),\Q_\ell/\Z_\ell(2))|^2$ is $1$ for almost all $\ell$, and so is 
$|H_\ur^3(k(X),\Q_\ell/\Z_\ell(2))|$. This proves Theorem \ref{conjecturetheorem} (a). 
Finally, we give a proof of the above Lemma \ref{sublemma1}:

\begin{proof}[Proof of Lemma \ref{sublemma1}]
We fix a prime number $\ell$. Since we assume $\boldsymbol{TB}^{\boldsymbol{2}}(X)$ and that $CH^2(X)$ is 
finitely generated, there is a following exact sequence
$$0\to CH^2(X)\otimes \Z_\ell \to H^4_\et(X,\Z_\ell(2)) \to G \to 0,$$
where $G$ is a group whose order is equal to $H_\ur^3(k(X),\Q_\ell/\Z_\ell(2))$ which is finite 
by our assumptions and Theorem \ref{finiteness} (a). Applying the $\Ext$-functor $\Ext^*_{\Z_\ell\text{-}\mod}(-,\Z_\ell)$ for the above sequence of $\Z_\ell$-modules, we get the following long exact sequence: 
$$0 \to \Hom_{\Z_\ell\text{-}\mod} (G,\Z_\ell) \to \Hom_{\Z_\ell\text{-}\mod} (H^4_\et(X,\Z_\ell(2)),\Z_\ell) \overset{\beta_{2,\ell}}{\longrightarrow} 
\Hom_{\Z_\ell\text{-}\mod} (CH^2(X) \otimes \Z_\ell,\Z_\ell)$$
$$\ \ \ \ \ \ \to \Ext_{\Z_\ell\text{-}\mod}^1(G,\Z_\ell) \to \Ext_{\Z_\ell\text{-}\mod}^1 (H^4_\et(X,\Z_\ell(2)),\Z_\ell) \to 
\Ext_{\Z_\ell\text{-}\mod}^1 (CH^2(X) \otimes \Z_\ell,\Z_\ell)\to 0.$$
We denote the map $\Hom_{\Z_\ell\text{-}\mod}(H_\et^4(X,\Z_\ell(2)),\Z_\ell) 
\to \Hom_{\Z_\ell\text{-}\mod}(CH^2(X)\otimes \Z_\ell, \Z_\ell)$ in the 
above sequence as $\beta_{2,\ell}$. 
This sequence of $\Ext$-groups breaks up at degree one, because 
$\Z_\ell$ has $\Ext$-dimension one.  
By the finiteness of $G$, $\Hom_{\Z_\ell\text{-}\mod}(G,\Z_\ell)=0$. 
Since the $\ell$-adic \'etale cohomology groups 
$H_\et^4(X,\Z_\ell(2))$ is finitely generated 
(Theorem \ref{DGconsequence} (1)), the quotient 
$H_\et^4(X,\Z_\ell(2))/H_\et^4(X,\Z_\ell(2))_\tors$ is free, 
{\em i.e.}, projective. For a projective $\Z_\ell$-module $M$, 
we have $\Ext^q_{\Z_\ell\text{-}\mod}(M,\Z_\ell)=0\ \ (q\geq 1)$, and have 
an isomorphism 
$$\Ext^1_{\Z_\ell\text{-}\mod}(M,\Z_\ell) \simeq \Hom(M_\tors,\Q_\ell/\Z_\ell).$$
Using this fact, one can rewrite $\Ext^1$-groups in the above long exact sequence 
as follows:

\medskip
$\displaystyle \Ext^1_{\Z_\ell\text{-}\mod}(G,\Z_\ell) \simeq \Hom_{\Z_\ell\text{-}\mod}(G,\Q_\ell/\Z_\ell),$

\medskip
$\displaystyle \Ext^1_{\Z_\ell\text{-}\mod}(H^4_\et(X,\Z_\ell(2)),\Z_\ell) 
\simeq \Hom_{\Z_\ell\text{-}\mod} (H^4_\et(X,\Z_\ell(2))_\tors,\Q_\ell/\Z_\ell),$

\medskip
$\displaystyle \Ext^1_{\Z_\ell\text{-}\mod}(CH^2(X)\otimes \Z_\ell,\Z_\ell) 
\simeq \Hom_{\Z_\ell\text{-}\mod}(CH^2(X)_{\ell\text{-}\tors},\Q_\ell/\Z_\ell).$

\medskip
\noindent
For a finite $\Z_\ell$-module $M$, 
we have $|\Hom_{\Z_\ell \text{-} \mod}(M,\Q_\ell/\Z_\ell)|=|M|$. 
Combining these facts, we finish the proof of Lemma \ref{sublemma1}, and that of Theorem 1 (a). 
\end{proof}

\subsection{Proof of the higher Chow formula} 
To obtain Theorem \ref{mainthm2}, 
we rewrite each term in Milne-Kahn-Chambert-Loir formula $(\natural)$ for $\zeta^*(X,2)$ 
in terms of higher Chow groups and unramified cohomology groups. 
By Lemma \ref{tech}, 
it remains to compute 
$H^6_\et(X,\hZ(2))$ and $(\clubsuit$). 
We have isomorphisms of finite groups 
$$H^6_\et(X,\hZ(2)) \underset{{\rm Theorem\ \ref{Milneduality}}}{\simeq} \left( H_\et^3(X,\Q/\Z(2)) \right)^* \simeq 
\left( H_\et^4(X,\hZ(2))_\tors \right)^*,$$
where the latter isomorphism follows from the 
following short exact sequence 
$$0 \to H^3_\et(X,\hZ(2))\otimes \Q/\Z \to H^3_\et(X,\Q/\Z(2)) \to H^4_\et(X,\hZ(2))_\tors \to 0$$and the finiteness of $H_\et^3(X,\Q/\Z(2))$ by Theorem \ref{DGconsequence} (3). 
$H_\et^4(X,\hZ(2))_\tors$ and $(\clubsuit)$ are obtained by Lemma \ref{sublemma1} and 
taking products with respect to all prime numbers $\ell$. Let $\beta_1:H_\et^5(X,\hZ(2)) \to \Hom_\cont(H_\et^4(X,\hZ(2)), \hZ)$ be the map induced by the pairing $H_\et^4(X,\hZ(2)) \times H_\et^5(X,\hZ(2)) \to H_\et^9(X,\hZ(2)) \simeq \hZ$ 
due to Milne (\cite[Lemma 5.3]{Milne2}), and $\beta_2$ is the induced map by the codimension two cycle class map $CH^2(X)\otimes \hZ \to H_\et^4(X,\hZ(2))$. We denote $\beta$ as a composition $\beta_2 \circ \beta_1$. 
By Milne's result (\cite[Lemma 5.3]{Milne2}), we see $\Coker(\beta)=\Coker(\beta_2)$. 
We have the following Key-Lemma:

\begin{lemma3}\label{keylemma}
Assume $\boldsymbol{TB}^{\boldsymbol{2}}(X)$ and that $CH^2(X)$ is finitely generated. 
Then $\beta$ and $\beta_2$ have finite cokernel, and the following equality holds$:$
$$|\Coker(\beta)|=|\Coker(\beta_2)|=\frac{|H_\ur^3(k(X),\Q/\Z(2))| \cdot |CH^2(X)_\tors|}{|H^4_\et(X,\hZ(2))_\tors|}.$$
\end{lemma3}

\begin{proof}
We have the following isomorphism:
$$\Coker(\beta)=\prod_{{\rm all\ primes\ }\ell} \Coker(\beta_{2,\ell}).$$
As we have seen in $\x$\ref{subsec:ell}, 
$|H_\ur^3(k(X),\Q_\ell/\Z_\ell(2))|$, $|CH^2(X)_\tors \otimes \Z_\ell|$ and 
$|H_\et^4(X,\Z_\ell(2))_\tors|$ are all $1$ for almost all $\ell$ under the assumptions of $\boldsymbol{TB}^{\boldsymbol{2}}(X)$ and the finite generation of $CH^2(X)$. 
Thus taking product of the equality in Lemma \ref{sublemma1} with respect to all prime numbers $\ell$ 
finishes our proof. 
\end{proof}

This completes the proof of our higher Chow formula.

\section{Comparison with Geisser's formula}\label{sec:comparison}

In this section, we compare our higher Chow formula (B) with Geisser's formula 
written in Weil-\'etale cohomology groups. 
Before starting our comparison arguments, 
we will recall basic properties of Weil-\'etale cohomology briefly. 

\subsection{Basics on Weil-\'etale motivic theory}\label{setup}\index{Weiletale@Weil-\'etale cohmology} 

Let $\Gamma = \Gal(\oF_q/\F_q)$ be the absolute Galois group 
of $\F_q$, and let $\Gamma_0\simeq \Z$ be 
the Weil subgroup of $\Gamma$ generated by the 
Frobenius element. Let $X$ be a smooth variety defined over $\F_q$, and 
let $n$ be a non-negative integer. 

Usual \'etale sheaves on $\bX$ arising from $X$ are equipped with $\Gamma$-action. 
A Weil-\'etale sheaf on $\bX$ is defined as an \'etale sheaf on $\bX$ equipped with $\Gamma_0$-action. 
If we denote the \'etale topos on $\bX$ by $\T_\et$, 
and the Weil-\'etale topos on $\overline{X}$ by $\T_\Wet$, then 
we have a 
morphism of topoi (cf. \cite[Expos\'e IV.3]{SGA4}): 
$$\gamma=(\gamma^*,\gamma_*):\ \xymatrix{
\T_\et \ar@<0.9ex>[r]^{\gamma^*}  & \T_\Wet \ar@<0.1ex>[l]^{\gamma_*}
},$$
where $\gamma^* : \T_\et \to \T_\Wet$ is the restriction functor. 

Using the \'etale motivic complex $\Z(n)$ defined in $\x$\ref{finitenessresult}, 
we define the Weil-\'etale motivic cohomology as the Weil-\'etale hypercohomology:
$$H_\Wet^i(X,\Z(n)) := \HH_\Wet^i(X,\gamma^* \Z(n)).$$
We also define $H_\Wet^i(X,A(n))$ for an abelian group $A$ for by 
tensoring to the \'etale motivic complex $\Z(n) \otimes A$. 
We are mainly concerned with case $A \in \{ \Z,\ \Q,\ \Z/m\Z,\ \Q/\Z \}$. 
By results of Geisser \cite[Theorem 7.1]{Geisser1}, Weil-\'etale motivic cohomology groups  are related with \'etale cohomology groups and higher Chow groups as follows:

\begin{theorem}[Geisser]\label{G-tech}
Let $X$ be a smooth variety defined over a finite field $\F_q$. 

\begin{enumerate}
\item[(1)] There is a long exact sequence 
$$ \cdots \to H_\et^i(X,\Z(n)) \to H_\Wet^i(X,\Z(n)) \to H_\et^{i-1}(X,\Q(n)) \to H_\et^{i+1}(X,\Z(n)) \to \cdots.$$

\item[(2)] For torsion coefficients, Weil-\'etale motivic cohomology is isomorphic to \'etale cohomology for all degrees, {\em i.e.,} we have the following for all $i$ and $n$$:$ 
$$H_\Wet^{i}(X,\Z/m\Z(n)) \simeq H_\et^i(X,\Z/m\Z(n)),$$
where $\Z/m\Z(n)$ is as in $\x$\ref{etale}.

\item[(3)] For rational coefficient, the following holds for all $i$ and $n$$:$ 
$$H_\Wet^i(X,\Q(n)) \simeq \bigg( CH^n(X,2n-i) \oplus CH^n(X,2n-i+1) \bigg) \otimes \Q.$$
\end{enumerate}
\end{theorem}

We will often use this theorem for $n=2$ case in the following subsections.

\subsection{Weil-\'etale cycle class map}\label{Weiletalecycle}\index{naive@(naive) Weil-\'etale cycle class map}

Throughout this paper, we will use a Weil-\'etale cycle class map 
defined as follows; we consider the following adjunction morphism of complexes: 
$$\Z(n) \to R\gamma_*\gamma^*\Z(n).$$
Applying the hypercohomology functor for this map, 
we get the following composite map: 
$$CH^n(X)\to H_\et^{2n}(X,\Z(n)) \to H_\Wet^{2n}(X,\Z(n)).$$ 
We denote this map by $\r_\Wet^n$, and call it {\em a Weil-\'etale cycle class map}.

\subsection{Two conjectures on Weil-\'etale motivic cohomology}\label{conjj}

There are two important conjectures on Weil-\'etale motivic cohomology. 
Let $X$ be a $d$-dimensional smooth projective variety defined over a finite 
field $\F_q$. 
\begin{enumerate}
\item ${\boldsymbol{L}^{\boldsymbol{n}}(X)}$: For every $i$, $H_\Wet^i(X,\Z(n))$ is finitely generated. 
\item ${\boldsymbol{K}^{\boldsymbol{n}}(X)}$: For every prime number $\ell$, Weil-\'etale motivic cohomology group is an integral model of the $\ell$-adic \'etale cohomology, 
{\em i.e.}, the following map induced by the isomorphism in Theorem \ref{G-tech} (2) 
is an isomorphism  
$$H_\Wet^i(X,\Z(n)) \otimes \Z_\ell \overset{\simeq}{\longrightarrow} H_\et^i(X,\Z_\ell(n)).$$
\end{enumerate}

\noindent
The first conjecture ${\boldsymbol{L}^{\boldsymbol{n}}(X)}$ is proposed by Lichtenbaum, 
and the second one ${\boldsymbol{K}^{\boldsymbol{n}}(X)}$ is proposed by Kahn (\cite[Conjecture 3.2]{Kahn4}). 
The second assumption in Theorem \ref{secondmainthm} in the introduction is nothing 
other than $\boldsymbol{L}^{\boldsymbol{2}}(X)$. 
By the work due to Geisser (\cite[Theorem 8.4]{Geisser1}), 
we have the following fact concerning the relations between these conjectures:

\begin{theorem}[Geisser]\label{Geisserconjecture}
The following implications hold$:$
$${\boldsymbol{K}^{\boldsymbol{n}}(X)}+{\boldsymbol{K}^{\boldsymbol{d-n}}(X)} \Rightarrow {\boldsymbol{L}^{\boldsymbol{n}}(X)} \Rightarrow {\boldsymbol{K}^{\boldsymbol{n}}(X)} \Rightarrow {\boldsymbol{TB}^{\boldsymbol{n}}(X)}.$$
Here ${\boldsymbol{TB}^{\boldsymbol{n}}(X)}$ denotes the Tate \& Beilinson conjecture stated in $\x$\ref{pre}. 
\end{theorem}

Consequently, if $\dim X=4$, then 
${\boldsymbol{L}^{\boldsymbol{2}}(X)}$ is equivalent to ${\boldsymbol{K}^{\boldsymbol{2}}(X)}$, 
and they imply $\boldsymbol{TB}^{\boldsymbol{2}}(X)$.  
In the rest of this section, we are mainly concerned with $4$-dimensional varieties $X$ satisfying $\boldsymbol{L}^{\boldsymbol{2}}(X)$.

\subsection{Weight two calculations}

Geisser calculated Weil-\'etale motivic cohomology in the case of weight $0$ and $1$ in \cite[Proposition 7.4]{Geisser1}. 
For our purpose, we need to calculate the weight $2$ case in terms of 
higher Chow groups. 
Throughout this subsection, let $X$ be a $4$-dimensional smooth projective variety 
over $\F_q$. 

\begin{lemma}\label{WetLemma}
There are isomorphisms for $i\geq 6$$:$ 
$$H_\Wet^{i}(X,\Q/\Z(2)) \underset{{\rm (a)}}{\simeq} H_\Wet^{i+1}(X,\Z(2)) \underset{{\rm (b)}}{\simeq} H_\et^{i+1}(X,\Z(2)) \underset{{\rm (c)}}{\simeq} H^i_\et(X,\Q/\Z(2)).$$
\end{lemma}

\begin{proof} 
By the exact sequence
$$ \cdots \to H_\et^i(X,\Z(2)) \to H_\et^i(X,\Q(2)) \to H_\et^i(X,\Q/\Z(2)) \to H^{i+1}_\et(X,\Z(2)) \to \cdots,$$
and the vanishing fact $H_\et^i(X,\Q(2))=0$ for $i>4$ (\cite[Theorem 3.6]{Geisser2}), 
we have the isomorphism (c). 
We show (a) and (b). 
Consider the long exact sequence 
of Weil-\'etale motivic cohomology groups attached to $\Z(2)\to \Q(2) \to \Q/\Z(2) \to \Z(2)[1]$:
$$\cdots \to H_\Wet^i(X,\Z(2)) \to H_\Wet^i(X,\Q(2)) \to H_\Wet^i(X,\Q/\Z(2)) \to H_\Wet^{i+1}(X,\Z(2)) \to \cdots.$$
By the results of Geisser (Theorem \ref{G-tech} (3)) and the fact that 
$H_\mot^i(X,\Q(2))=0$ for $i\geq 5$, we have (a). 
(b) follows from Geisser's long exact sequence 
(Theorem \ref{G-tech} (1))
$$ \cdots \to H_\et^i(X,\Z(2)) \to H_\Wet^i(X,\Z(2)) \to H_\et^{i-1}(X,\Q(2)) \to H_\et^{i+1}(X,\Z(2))\to \cdots,$$
and again the vanishing fact concerning $H_\et^i(X,\Q(2))$. 
\end{proof}
\noindent
By a vanishing theorem of Geisser (\cite[Theorem 7.3]{Geisser1}), 
we have $H_\Wet^i(X,\Z(2))=0$ for $i\geq 10$, therefore 
the groups in Lemma \ref{WetLemma} are zero.

\begin{lemma}\label{1stcorr}
There are isomorphisms for $i=1,2,3$ 
$$H_\Wet^{i+6}(X,\Z(2)) \simeq \left( CH^2(X,i)_\tors \right)^*.$$
\end{lemma}

\begin{proof}
These are consequences of Lemma 
\ref{tech} and Lemma \ref{WetLemma}:
\begin{eqnarray*}
H_\Wet^{i+6}(X,\Z(2)) &\underset{{\rm Lemma\ }\ref{WetLemma}\ {\rm (b),(c)}}{\simeq}& 
H_\et^{i+5}(X,\Q/\Z(2)) \underset{{\rm Theorem\ }\ref{Milneduality}}{\simeq} 
\left( H_\et^{4-i}(X,\hZ(2)) \right)^* \\
&\underset{{\rm Lemma\ }\ref{tech}}{\simeq}& 
\left( CH^2(X,i) \right)^*\ \ \ \ {\rm for\ }i=1,2,3, 
\end{eqnarray*}
where we used the assumption $i=1,2,3$ to apply 
Lemma \ref{WetLemma} and Lemma \ref{tech}.
\end{proof}

\begin{lemma}\label{5th}
Assume $CH^2(X,1)\otimes \Q=0$. Then 
we have $H_\et^5(X,\Z(2)) \simeq H_\Wet^5(X,\Z(2))_\tors$. 
\end{lemma}

\begin{proof}
We consider the following commutative diagram with exact columns: 
$$\xymatrix{
H_\et^4(X,\Q(2)) \ar[d] \ar[r]^{{\rm (1)}\ \ } & H_\Wet^4(X,\Q(2)) \ar[d] \\
H_\et^4(X,\Q/\Z(2)) \ar@{->>}[d] \ar[r]_{{\rm (2)}\ \ }^{\simeq\ \ } & H_\Wet^4(X,\Q/\Z(2)) \ar@{->>}[d] \\
H_\et^5(X,\Z(2))_\tors \ar[r]_{{\rm (3)}\ \ } & H_\Wet^5(X,\Z(2))_\tors
.}$$

\noindent
The right vertical sequences in this diagram is a 
long exact sequence stated in the proof of Lemma \ref{WetLemma}. 
We have the following isomorphisms:
$$H_\et^4(X,\Q(2)) \underset{{\rm (d)}}{\simeq} CH^2(X)\otimes \Q 
\underset{{\rm (e)}}{\simeq} H_\Wet^4(X,\Q(2)).$$
(d) is a result of Geisser (\cite[Proposition 3.6]{Geisser2}), and (e) follows from 
Theorem \ref{G-tech} (3) and the 
assumption $CH^2(X,1)\otimes \Q=0$. 
As a consequence, the top horizontal arrow (1) is an isomorphism. 
The middle arrow (2) is an isomorphism by Theorem \ref{G-tech} (2). 
Thus the bottom arrow (3) is an isomorphism. Since we know that 
$H_\et^5(X,\Z(2))$ is a torsion group, $H_\et^5(X,\Z(2))_\tors=H_\et^5(X,\Z(2))$. 
This completes our proof.
\end{proof}


\begin{lemma}\label{2ndcorr} Assume ${\boldsymbol{TB}^{\boldsymbol{2}}(X)}$. Then 
$H_\Wet^i(X,\Z(2))_\tors$ is finite for all $i$, and 
the equality $$\left| H_\Wet^i(X,\Z(2))_\tors \right|=\left|H_\Wet^{10-i}(X,\Z(2))_\tors \right|$$holds for all $i$. 
\end{lemma}

\begin{proof}
Since we assume $\boldsymbol{TB}^{\boldsymbol{2}}(X)$, 
we have the following perfect pairing of finite groups
$$H_\Wet^i(X,\Z(2))_\tors \times H_\Wet^{10-i}(X,\Z(2))_\tors \to \Q/\Z$$
by a result of Kahn (\cite[Corollaire 3.8]{Kahn}). 
This pairing gives the desired equalities. 
\end{proof}

\subsection{Proof of Theorem \ref{secondmainthm}}\label{comparison}

Our task is to prove the following equality: 
$$\frac{1}{R_1} \left| H_\ur^3(k(X),\Q/\Z(2)) \right|^2 \cdot \left| CH^2(X)_\tors \right|^2 
=\frac{1}{R_2} \left| H_\Wet^4(X,\Z(2))_\tors \right| \cdot \left| H_\Wet^6(X,\Z(2))_\tors \right|,$$
where $R_1$ and $R_2$ are as we defined in the introduction. 
We investigate the relation between $R_1$ with $R_2$. 
The pairing stated in the introduction 
$$H_\Wet^{4}(X,\Z(2)) \times H_\Wet^4(X,\Z(2)) \to H_\Wet^8(X,\Z(4)) \overset{- \cup e}{\to} H_\Wet^9(X,\Z(2)) \overset{\deg}{\to} \Z$$
induces the following map:
$$\displaystyle H_\Wet^4(X,\Z(2)) \overset{\vartheta}{\to} \Hom(H_\Wet^4(X,\Z(2)),\Z).$$
Using the Weil-\'etale cycle class map $\r_\Wet^2$, 
we consider the following commutative diagram of 
pairings:
$$\xymatrix{
CH^2(X) \times CH^2(X) \ar[rrrr] \ar[d]_{\r_\Wet^2 \times \r_\Wet^2} & & & & \Z \ar@{=}[d] \\
H_\Wet^4(X,\Z(2)) \times H_\Wet^4(X,\Z(2)) \ar[rr] &  & H_\Wet^8(X,\Z(2)) \ar[r]_{-\cup e} & H_\Wet^9(X,\Z(2)) \ar[r] & \Z.
}$$

\noindent
By Theorem \ref{finiteness} (c), the pairing $CH^2(X)\times CH^2(X) \to \Z$ 
is non-degenerate when tensored with $\Q$. 
Since $H_\Wet^4(X,\Z(2))$ is also assumed to be finitely generated, 
this pairing is non-degenerate when tensored with $\Q$ 
(see the proof of \cite[Theorem 9.1]{Geisser1}). 
(The assumption that $H_\Wet^4(X,\Z(2))$ is finitely generated is ${\boldsymbol{L}^{\boldsymbol{2}}(X)}$ 
mentioned in $\x$\ref{conjj}. 
We remind that ${\boldsymbol{L}^{\boldsymbol{2}}(X)}$ implies ${\boldsymbol{TB}^{\boldsymbol{2}}(X)}$.) 
By Theorem \ref{TBconsequence} (3) and Theorem \ref{G-tech} (3), 
the Weil-\'etale cycle class map $\r_\Wet^2$ is bijective when 
tensored with $\Q$, and $\r_\Wet^2$ has finite kernel and cokernel. 
By the perfectness of both pairings, the map $\r'$ has also finite kernel and cokernel. 
The above diagram of pairings induce the following diagram:
$$\xyoption{curve}
\xymatrix{
CH^2(X) \ar[d]^{\r_\Wet^2} \ar[rr]^{\theta} & & \Hom(CH^2(X),\Z) \\
H_\Wet^4(X,\Z(2)) \ar[r]_{-\cup e} \ar@/_1.5pc/[rr]_{\vartheta} & H_\Wet^5(X,\Z(2)) \ar[r] & \Hom(H_\Wet^4(X,\Z(2)),\Z) 
\ar[u]_{\r'}
,}$$
where $\r'$ is the induced map by $\r_\Wet^2$, and 
$\theta$ has finite kernel and cokernel (Theorem \ref{finiteness} (c)). 
The commutativity of the this diagram follows 
from Milne's theorem (\cite[Theorem 5.4]{Milne2}). 
By these facts, the maps $\r_\Wet^2,\ \r', \theta$ and $\vartheta$ have finite 
kernels and cokernels. Hence we have 
the following equality:
$$\frac{| \Ker (\theta) |}{| \Coker (\theta) |} = 
\frac{| \Ker (\r_\Wet^2)|}{| \Coker (\r_\Wet^2) |} \cdot 
\frac{| \Ker (\vartheta) |}{| \Coker (\vartheta) |} \cdot 
\frac{| \Ker (\r') |}{| \Coker (\r') |}.$$

We know the followings about $\theta$ and $\vartheta$:
\begin{enumerate}
\item[(i)] $|\Coker(\theta)|=R_1$ and $|\Coker(\vartheta)|=R_2$ by definition. 
\item[(ii)] $\Ker(\theta)=CH^2(X)_\tors$ and $\Ker(\vartheta)=H_\Wet^4(X,\Z(2))_\tors$. 
\end{enumerate}

It remains to compute kernels and cokernels of $\r_\Wet^2$ and $\r'$. 

\begin{lemma}\label{SSlemma} 
Assume $\boldsymbol{L}^{\boldsymbol{2}}(X)$, that $CH^2(X)$ is finitely generated, and that $CH^2(X,1)\otimes \Q=0$. 
Then 

\begin{enumerate}
\item[(1)] The Weil-\'etale cycle class map $\r_\Wet^2$ is injective. 

\item[(2)] The following equality holds$:$ $|\Coker(\r_\Wet^2)|=|H_\ur^3(k(X),\Q/\Z(2))|$.
\end{enumerate}
\end{lemma}

\begin{proof}
(1) By the assumption $CH^2(X,1)\otimes \Q=0$, 
the Weil-\'etale cycle class map $\r_\Wet^2 \otimes \Q$ is 
bijective. Hence we need to show that $\r_\Wet^2$ is injective on torsions. 
Consider the following diagram:
$$\xyoption{curve}
\xymatrix{
CH^2(X) \ar[r]^{\r_\Wet^2\ \ \ } \ar@/_1.5pc/[rr]_{\r_{\Z_\ell}^2\ } & H_\Wet^4(X,\Z(2)) \ar[r] & H_\et^4(X,\hZ(2)).}$$
By Theorem \ref{TBconsequence} (2), 
$\r_{\Z_\ell}^2$ is injective on torsion, and therefore $\r_\Wet^2$ is injective on torsion.

(2) We consider the following diagram for each prime number $\ell$:
$$\xyoption{curve}
\xymatrix{
CH^2(X) \otimes \Z_\ell \ar[r]^{\r_\Wet^2\otimes \Z_\ell\ \ \ \ } \ar@/_1.5pc/[rr]_{\r_{\Z_\ell}^2\ } & H_\Wet^4(X,\Z(2)) \otimes \Z_\ell \ar[r]^{\ \ \ \simeq}_{\ \ \ (\diamondsuit)} & H_\et^4(X,\Z_\ell(2)),}$$
where the map $(\diamondsuit)$ is an isomorphism 
because we assumed $\boldsymbol{K}^{\boldsymbol{2}}(X)$, 
which is equivalent to $\boldsymbol{L}^{\boldsymbol{2}}(X)$ by 
Theorem \ref{Geisserconjecture}. 
We have an isomorphism 
$\Coker(\r_\Wet^2\otimes \Z_\ell) \simeq \Coker (\r_{\Z_\ell}^2)$. 
By a result of Saito-Sato (\cite[Proposition 4.3.5]{SaitoSato}), 
the equality $|\Coker(\r_{\Z_\ell}^2)|=|H_\ur^3(k(X),\Q_\ell/\Z_\ell(2))|$ holds, which 
is $1$ for almost all $\ell$ by Theorem \ref{conjecturetheorem} (a). 
Therefore, combining for all $\ell$, and by $\boldsymbol{K}^{\boldsymbol{2}}(X)$ 
(Weil-\'etale motivic cohomology group is an integral lattice of 
\'etale cohomology group),  
we have the desired equality 
$|\Coker(\r_\Wet^2)|=|H_\ur^3(k(X),\Q/\Z(2))|$. 
\end{proof}

Finally, we investigate the case of $\r'$: 

\begin{lemma3}\label{keylemma2}
Assume $\boldsymbol{L}^{\boldsymbol{2}}(X)$, that $CH^2(X)$ is finitely generated, and that $CH^2(X,1)\otimes \Q=0$. 
Then 

\begin{enumerate}
\item[(1)] $\r'$ is injective. 

\item[(2)] We have  
$\displaystyle |\Coker (\r')|=\frac{|H_\ur^3(k(X),\Q/\Z(2))|\cdot |CH^2(X)_\tors|}{|H_\Wet^4(X,\Z(2))_\tors|}.$
\end{enumerate}
\end{lemma3}

\begin{proof}
We prove this lemma in the same way as in the proof of Key-Lemma \ref{keylemma}.  
Apply the $\Ext$-functor $\Ext^i_{\Z\text{-}\mod}(-,\Z)$ to the following short 
exact sequence:
$$0 \to CH^2(X) \overset{\r_\Wet^2}{\longrightarrow} H_\Wet^4(X,\Z(2)) \to G \to 0,$$
where $\r_\Wet^2$ is injective by Lemma \ref{SSlemma} (1), 
and $G$ is finite satisfying $|G|=|H_\ur^3(k(X),\Q/\Z(2))|$ by Lemma \ref{SSlemma} (2). 
We have used the all assumptions here.

(1) The injectivity of $\r'$ immediately follows from $\Hom(G,\Z)=0$. 

(2) By similar arguments as for Key-Lemma \ref{keylemma}, 
we now have the following exact sequence: 
$$0 \to \Hom(H_\Wet^4(X,\Z(2)),\Z) \overset{\r'}{\to} \Hom(CH^2(X),\Z) \to \Hom(G,\Q/\Z)\ \ \ \ \ \ \ \ \ $$ 
$$\ \ \ \ \ \ \ \ \ \ \ \ \ \ \ \to \Hom (H_\Wet^4(X,\Z(2))_\tors,\Q/\Z) \to \Hom(CH^2(X)_\tors, \Q/\Z) \to 0,$$
which implies the assertion. 
\end{proof}

This completes the proof of Theorem \ref{secondmainthm}.




\end{document}